\documentclass[10pt, article]{amsart}
\usepackage{geometry} 
\usepackage{ae} % or {zefonts}
\usepackage[T1]{fontenc}
\usepackage[cp1250]{inputenc}
\usepackage{amsmath}
\usepackage{amssymb, amsfonts,amscd,verbatim}
\usepackage{MnSymbol}
\usepackage{mathrsfs} % for \mathscr
\usepackage{xypic}

\usepackage[normalem]{ulem}
\usepackage{hyperref}
\usepackage{indentfirst}
\usepackage{latexsym}
\input xy
\xyoption{all}

\usepackage{amsmath}    % need for subequations
\theoremstyle{plain}
\newtheorem{Pocz}{Poczatek}[section]
\newtheorem{Proposition}[Pocz]{Proposition}

\newtheorem{Theorem}[Pocz]{Theorem}
\newtheorem{Corollary}[Pocz]{Corollary}

\newtheorem{Lemma}[Pocz]{Lemma}
\newtheorem{Observation}[Pocz]{Observation}

\newtheorem{Question}[Pocz]{Question}

\newtheorem{Example}[Pocz]{Example}
\theoremstyle{definition}
\newtheorem{Definition}[Pocz]{Definition}

\theoremstyle{remark}
\newtheorem{Remark}[Pocz]{Remark}

\DeclareMathOperator*{\st}{st}
\numberwithin{equation}{section}

\author{Jerzy~Dydak}
\address{University of Tennessee, Knoxville, USA}
\email{jdydak@utk.edu}
\author{Thomas ~ Weighill}
\address{University of Tennessee, Knoxville, USA}
\email{tweighil@vols.utk.edu}

\title{Extension theorems for large scale spaces via coarse neighbourhoods}

\date{ \today
}
\keywords{large scale space; hybrid large scale normality; neighbourhood operator; Higson compactification; slowly oscillating map \\ \copyright 2018. This manuscript version is made available under the CC-BY-NC-ND 4.0 license http://creativecommons.org/licenses/by-nc-nd/4.0/. Published v. at https://doi.org/10.1007/s00009-018-1106-z}


\begin{document}
%\fontsize{18}{20pt}\selectfont

\begin{abstract}
We introduce the notion of (hybrid) large scale normal space and prove coarse geometric analogues of Urysohn's Lemma and the Tietze Extension Theorem for these spaces, where continuous maps are replaced by (continuous and) slowly oscillating maps. To do so, we first prove a general form of each of these results in the context of a set equipped with a neighbourhood operator satisfying certain axioms, from which we obtain both the classical topological results and the (hybrid) large scale results as corollaries. We prove that all metric spaces are hybrid large scale normal, and characterize those locally compact abelian groups which (as hybrid large scale spaces) are hybrid large scale normal. Finally, we look at some properties of the Higson compactifications and coronas of hybrid large scale normal spaces.
\end{abstract}

\maketitle

\section{Introduction}
Large scale geometry (also called coarse geometry) is concerned with studying the large scale behaviour of spaces. This typically means studying properties which are invariant not under homeomorphism or homotopy equivalence of spaces, but rather under so-called ``coarse equivalence'' of spaces. The main applications of this theory are found in geometric group theory (see for example \cite{Gromov93}) and index theory \cite{RoeIndex}. For an introduction to the basic notions in coarse geometry we refer the reader to \cite{NowakYu} or \cite{Roe}. All necessary definitions will be recalled in this paper.

Usually, topology does not play a role when studying the large scale properties of spaces. Indeed, it is easy to check  that every metric space is coarsely equivalent to a discrete metric space. On the other hand, the literature shows that it is useful to introduce concepts in large scale geometry which are analogous to well-known concepts in classical topology. For example, asymptotic dimension (introduced by Gromov \cite{Gromov93}) is the large scale analogue of covering dimension. Consider now the following fundamental result in general topology.

\begin{Theorem}[Tietze Extension Theorem]
Let $X$ be a normal topological space and let $A$ be a closed subset of $X$. Then any continuous function $f: A \rightarrow [0,1]$ extends to a continuous function $g: X \rightarrow [0,1]$.
\end{Theorem}

When studying the large scale properties of spaces, one typically replaces continuous functions to subsets of $\mathbb{R}^n$ with slowly oscillating functions to subsets of $\mathbb{R}^n$. A good example of this analogy in action can be found in \cite{DM}, where asymptotic dimension is approached by considering extensions of slowly oscillating functions to spheres, based on the approach to covering dimension via extensions of continuous maps to spheres. In the same paper, the authors also prove the following result:
 
 \begin{Theorem}[Dydak-Mitra \cite{DM}]\label{DMtheorem}
 Given a metric space $X$, any slowly oscillating function on a subset of $X$ to $[0,1]$ extends to a slowly oscillating function on the whole of $X$ to $[0,1]$.
 \end{Theorem}
 
This may be seen as a large scale Tietze Extension Theorem for metric spaces. In this paper, we prove a more general version of this result for what we call large scale normal spaces. In fact, we work in the more general context of hybrid large scale spaces introduced in \cite{ADH} -- that is, sets equipped with a large scale structure and a topology satisfying a compatibility axiom -- so that we are interested in maps which are both slowly oscillating and continuous. Results for large scale spaces (with no topology) can be recovered as special cases of the hybrid results by endowing the large scale space with the discrete topology.

To make the analogy with the classical topological situation clear, we prove a version of Urysohn's Lemma and the Tietze Extension Theorem for sets equipped with a neighbourhood operator satisfying certain axioms. This approach via abstract neighbourhood operators may be of independent interest to readers outside of coarse geometry. From the general versions of these results, we obtain both the classical topological results and the new results for hybrid large scale spaces. We then study (hybrid) large scale normal spaces further and give some examples of spaces which are hybrid large scale normal as well as some examples which aren't. As suggested by the above theorem of Dydak and Mitra, all metric spaces are large scale normal. We also look at hybrid large scale normality for certain special types of ls-structures, such as those arising from a group structure or induced by a metrizable compactification. Finally, we look at some applications of this theory to the Higson compactification, a compactification which captures large scale behaviour of a space, and the Higson corona.

\section{Hybrid large scale spaces}
The main context for the results in this paper is that of a hybrid large scale space, which is a set equipped with a topology (representing the small scale) and a large scale structure (or ls-structure) which are compatible in a suitable sense. The idea to consider a space equipped with a topology and ls-structure which are compatible goes back to Roe (see Chapter 2 of \cite{Roe}). Note that the notion of ls-structure is equivalent to the notion of coarse structure in the sense of \cite{Roe} (see \cite{DH}). 

Let us recall the definition of ls-structure from \cite{DH}. Let $X$ be a set. Recall that the \textbf{star} $\st(B, \mathcal{U})$ of a subset $B$ of $X$ with respect to a family $\mathcal{U}$ of subsets of $X$ is the union of those elements of $\mathcal{U}$ that intersect $B$. More generally, for two families $\mathcal{B}$ and $\mathcal{U}$ of subsets of X, $\st(\mathcal{B}, \mathcal{U})$ is the family $\{\st(B, \mathcal{U}) \mid B \in \mathcal{B}\}$.

\begin{Definition}
A \textbf{large scale structure} $\mathcal{L}$ on a set $X$ is a nonempty
collection of families of subsets of $X$ (which we call the \textbf{uniformly bounded families}) satisfying the following conditions:
\begin{itemize}
\item[(1)] $\mathcal{B}_1 \in \mathcal{L}$ implies $\mathcal{B}_2 \in \mathcal{L}$ if each element of $\mathcal{B}_2$ consisting of more than one point is contained in some element of $\mathcal{B}_1$.
\item[(2)] $\mathcal{B}_1, \mathcal{B}_2 \in \mathcal{L}$ implies $\st(\mathcal{B}_1, \mathcal{B}_2) \in \mathcal{L}$.
\end{itemize}
A set equipped with a large scale structure will be called an \textbf{ls-space}. A uniformly bounded family which is a cover is also called a \textbf{scale}.
\end{Definition}
 
\begin{Definition}
A \textbf{hybrid large scale space} (or \textbf{hls-space} for short) is a set $X$ equipped with both a large scale structure and a topology (which together we call the hybrid large scale structure on $X$) such that there is a uniformly bounded cover of $X$ which consists of open sets. We call a uniformly bounded open cover an \textbf{open scale}.
\end{Definition}

Note that in an hls-space, any scale can be coarsened to an open scale (we say that $\mathcal{U}$ \textbf{coarsens} $\mathcal{V}$ in case $\mathcal{V}$ refines $\mathcal{U}$).

\begin{Lemma}\label{FirstLemma}
Let $X$ be a set equipped with an ls-structure and a topology. Then the following are equivalent:
\begin{itemize}
\item[(1)] $X$, together with the two structures, gives an hls-space.
\item[(2)] there is a uniformly bounded cover $\mathcal{U}$ relative to the ls-structure such that for every subset $A$ in $X$, $\mathsf{cl}(A) \subseteq \st(A, \mathcal{U})$. 
\end{itemize}
Moreover, if $X$ is an hls-space, then for any open scale $\mathcal{U}$ and any subset $A$ of $X$, we have $\mathsf{cl}(A) \subseteq \st(A, \mathcal{U})$.
\end{Lemma}
\begin{proof}
To prove the last statement, notice that for any open scale $\mathcal{U}$ and subset $A$ of an hls-space $X$, if $x \in \mathsf{cl}(A)$ and $x \in U \in \mathcal{U}$, then $U$ intersects $A$. This also gives (1) $\Rightarrow$ (2).

(2) $\Rightarrow$ (1): Let $\mathcal{U}$ be as in (2). Then the interiors of the elements of the cover $\st(\mathcal{U}, \mathcal{U})$ form an uniformly bounded open family $\mathcal{V}$ of subsets. Moreover, if $U \in \mathcal{U}$, then $\mathsf{cl}(X \setminus \st(U, \mathcal{U})) \subseteq \st(X \setminus \st(U, \mathcal{U}), \mathcal{U}) \subseteq X \setminus U$, so $\mathcal{V}$ coarsens $\mathcal{U}$, and thus is a cover as required.
\end{proof}

\begin{Example}
The canonical example of a large scale space is a metric space $(X, d)$ equipped with the ls-structure consisting of all families $\mathcal{U}$ of subsets which refine $\{B(x, R) \mid x\in X\}$ for some $R > 0$. In fact, this ls-structure together with the metric topology gives an hls-space.
\end{Example} 

Every result in this paper which is proved for hls-spaces provides a version of the result for ls-spaces as a special case. This is because any ls-space  can be viewed as an hls-space by equipping it with the discrete topology. On the other hand, every topological space can be viewed as an hls-space by equipping it with the ls-structure consisting of all families of subsets. 

Let us now consider the notion of connectedness in the context of (hybrid) large scale spaces. In any scale category (see \cite{ADH}) an important issue is connectedness at some scale, that is, the existence of a scale such that any two points in $X$ are connected by a chain in that scale.

\begin{Definition}\label{U-componentDef}
 Given a cover $\mathcal{U}$ of a set $X$ and two points $x$ and $y$ in $X$, we say that $x$ and $y$ are $\mathcal{U}$-connected and write $x\sim_{\mathcal{U}} y$ if there is a finite sequence $U_i$, $1\leq i\leq n$, of elements of $\mathcal{U}$ such that $U_i\cap U_{i+1}\ne \emptyset$ for all $i < n$, $x\in U_1$ and $y \in U_n$. A \textbf{$\mathcal{U}$-component} of $X$ is an equivalence class of the equivalence relation $\sim_\mathcal{U}$. We say that $X$ is \textbf{$\mathcal{U}$-connected}, or \textbf{connected at the scale $\mathcal{U}$}, if it has at most one $\mathcal{U}$-component.
\end{Definition}

In the case of ls-spaces one is often interested in spaces that are $\mathcal{U}$-connected for some uniformly bounded cover $\mathcal{U}$ (for example, every geodesic metric space, as an ls-space, is such). This is not to be confused with the weaker condition, called \textbf{coarse connectedness} by Roe \cite{Roe}, which, in terms of uniformly bounded covers, can be translated as saying that any two points are $\mathcal{U}$-connected for some uniformly bounded cover $\mathcal{U}$. Clearly this is the same as to say that all finite subsets of $X$ are bounded (a subset of an ls-space is called \textbf{bounded} if it is an element of some uniformly bounded cover). 

\begin{Definition}\label{CoarseComponentDef}
A \textbf{coarse component} of a point $x$ in an ls-space is the union of all the bounded sets containing $x$.
\end{Definition}

A non-empty ls-space $X$ is thus coarsely connected in the sense of Roe if it has only one coarse component. To distinguish Roe's version of connectivity from the stronger one we introduce the following concept:

\begin{Definition}\label{ScaleConnectedDef}
An ls-space is \textbf{scale connected} if it is $\mathcal{U}$-connected for some uniformly bounded cover $\mathcal{U}$.
\end{Definition}

For example, the subspace $\{x^2 \mid  x \in \mathbb{N} \}$ with the ls-structure induced by the usual metric, is coarsely connected but not scale connected. In an hls-space, the topology of $X$ dictates large scale connectivity:

\begin{Proposition}
Let $X$ be an hls-space and let $\mathcal{U}$ be an open scale. Then
\begin{itemize}
\item[(1)] the $\mathcal{U}$-components of $X$ are open-closed,
\item[(2)] the coarse components of $X$ are open-closed,
\item[(3)] if the topology of $X$ is connected, then $X$ is connected at all open scales.
\end{itemize} 
\end{Proposition}
\begin{proof}
Given an open scale $\mathcal{U}$, the $\mathcal{U}$-components of $X$ are clearly open and they partition $X$, which proves (1). Part (2) follows from the fact that every bounded set is contained in an open bounded set. Finally, (3) follows from (1).
\end{proof}

\begin{Lemma}\label{DecompositionLemma}
 If $X$ is a hybrid large scale space and $\mathcal{U}=\{U_s\}_{s\in S}$ is an open scale of $X$, then 
 each $\mathcal{U}$-component of $X$ can be expressed as a union $\bigcup_{n=1}^\infty A_n$, where the sequence $\{A_n\}$ satisfies the following
 properties:
 \begin{enumerate}
 \item each $A_n$ is closed and bounded,
 \item $A_n$ is contained in the interior of $A_{n+1}$ for each $n\ge 1$.
 \end{enumerate}
\end{Lemma}
\begin{proof}
Let $A$ be a $\mathcal{U}$-component of $X$, and let $U_s \in \mathcal{U}$ be contained in $A$. Consider the sequence $B_n$ defined as follows:
\begin{enumerate}
\item[1.] $B_{1}= U_s$,
\item[2.] $B_{n+1}=st(B_{n},\mathcal{U})$.
\end{enumerate}
By the definition of $\mathcal{U}$-component, the union $\bigcup_{n=1}^\infty B_n$ is the whole of $A$. Define $A_n = \mathrm{cl}(B_n)$ for each $n$. Then the sequence $A_n$ satisfies the conditions by Lemma \ref{FirstLemma}. 
\end{proof}

\begin{Lemma}\label{precompactisbounded}
 If $X$ is a hybrid large scale space that is coarsely connected, then all precompact subsets (that is, subsets whose closure is compact) of $X$ are bounded.
\end{Lemma}
\begin{proof}
This is an easy consequence of Lemma \ref{DecompositionLemma} and the fact that in any coarsely connected ls-space, the finite union of bounded sets is bounded.
\end{proof}

\begin{Definition}
A subset $K$ of an ls-space $X$ is called \textbf{weakly bounded} if its intersection with any coarse component is bounded.
\end{Definition}

We now recall the definition of a slowly oscillating map. For a coarsely connected ls-space $X$, a slowly oscillating map is usually defined as a map to a metric space $M$ such that for every uniformly bounded family $\mathcal{U}$ in $X$ and every $\varepsilon > 0$ there is a bounded set $K$ in $X$ such that $(U \in \mathcal{U}) \wedge (U \cap K = \varnothing) \implies \mathsf{diam}(f(U)) < \varepsilon$. If $X$ is not coarsely connected, then this definition is too restrictive. Indeed, it is easy to check that, under this definition, a slowly oscillating map must be constant on all but one of the coarse components of $X$. Thus we use the following definition taken from \cite{ADH}, which reduces to the usual definition when $X$ is coarsely connected, and agrees with the classical definition of Higson function in \cite{Roe} for proper metric spaces (or more generally for proper hls-spaces introduced in the next section). 

\begin{Definition}\label{slowdef}
Let $X$ be an  ls-space, $M$ a metric space and $f: X \rightarrow M$ a map. Then $f$ is \textbf{slowly oscillating} if for every uniformly bounded family $\mathcal{U}$ in $X$ and every $\varepsilon > 0$ there is a weakly bounded set $K$ in $X$ such that $(U \in \mathcal{U}) \wedge (U \cap K = \varnothing) \implies \mathsf{diam}(f(U)) < \varepsilon$. 
\end{Definition}

Note that under Definition \ref{slowdef}, a map from an ls-space $X$ is slowly oscillating if and only if its restriction to each coarse component is slowly oscillating. 

\section{Proper hls-spaces}
By Lemma \ref{precompactisbounded}, every precompact subset of a coarsely connected hls-space is bounded. On the other hand, even for metric spaces with the induced hls-structure, it may not be the case that all bounded sets are precompact.

\begin{Definition}
A hybrid large scale space $X$ is called \textbf{proper} if its topology is Hausdorff, and its family of bounded sets is identical with the family of all precompact subsets of $X$. In particular, $X$ is (topologically) locally compact and coarsely connected.
\end{Definition}

For example, any proper metric space (i.e.~in which bounded sets are precompact) together with the induced hls-structure is a proper hls-space. It might initially appear that the notion of a proper hybrid large scale space is a generalization of the notion of coarsely connected \textbf{proper coarse space} introduced by Roe
(see \cite{Roe}, Definition 2.35), since the assumption of paracompactness is missing in our definition. However, Corollary \ref{ParacompactnessIssue} below shows that a proper hls-space $X$ must be paracompact, so the two notions are, in fact, identical. 

\begin{Corollary}\label{ParacompactnessIssue}
 The topology of any proper hybrid large scale space is paracompact.
\end{Corollary}
\begin{proof}
Pick an open scale $\mathcal{U}$ and express each $\mathcal{U}$-component of $X$ as in Lemma \ref{DecompositionLemma}. Suppose $\mathcal{V}$ is an open cover of $X$. Since each $A_n$ is closed and compact, we may suppose that each $A_n$ intersects only finitely many elements of $\mathcal{V}$. Moreover, each $A_n$ is paracompact since it is compact Hausdorff. Pick a partition of unity on $A_1$ subordinate to the cover $\{ V \cap A_1 \mid V \in \mathcal{V}  \}$. By Theorem 1.5 in \cite{DyPU}, we can extend this to a partition of unity on $A_2$ which is subordinate to the cover $\{V \cap A_2 \mid V \in \mathcal{V} \}$. Inductively we obtain a partition of unity on the whole of $X$ subordinate to the cover $\mathcal{V}$, where the continuity follows from the fact that each $A_n$ is contained in the interior of $A_{n+1}$.
\end{proof}

\begin{Corollary}
 There is no proper hybrid large scale structure on the space of all countable ordinals $S_\Omega$ whose topology is the order topology.
\end{Corollary}
\begin{proof}
$S_\Omega$  with the order topology is the basic example of a normal space that is not paracompact (see \cite{Mu}).
\end{proof}

We also have the following corollary of Lemma \ref{DecompositionLemma}. 

\begin{Corollary}
If $X$ is a proper hls-space and $\mathcal{U}$ is an open scale, then each $\mathcal{U}$-component admits a countable basis of bounded sets, that is, a countable set $\mathcal{B}$ of bounded sets such that every bounded set is contained in some element of $\mathcal{B}$. In particular, each $\mathcal{U}$-component of $X$ is $\sigma$-compact.
\end{Corollary}

If $f: X \rightarrow Y$ is a map from an ls-space $X$ to an ls-space $Y$, we say that $f$ is \textbf{large-scale continuous} or \textbf{ls-continuous} if for every uniformly bounded family $\mathcal{U}$ in $X$, the family
$$
f(\mathcal{U}) = \{f(U) \mid U \in \mathcal{U} \}
$$
is uniformly bounded in $Y$. A map $f: X \rightarrow Y$ between hls-spaces is called \textbf{hls-continuous} if it is continuous with respect to the topologies and ls-continuous with respect to the ls-structures. Two ls-continuous maps $f,g: X \rightarrow Y$ are said to be \textbf{close} if the family $\{ \{ f(x), g(x)\} \mid x \in X\}$ is uniformly bounded.

Recall that an ls-continuous map $f: X \rightarrow Y$ between ls-spaces is called a \textbf{large scale equivalence} (or coarse equivalence) if there is an ls-continuous map $f'$ in the other direction such that $ff'$ and $f'f$ are both close to the identity, i.e.~such that both families
$$
\{ \{ff'(y), y\} \mid y \in Y \}, \ \{ \{f'f(x), x\} \mid x \in X \} 
$$
are uniformly bounded. It is easy to check that an ls-continuous map $f: X \rightarrow Y$ is a large scale equivalence if and only if both of the following hold:

\begin{itemize}
\item $f$ is \textbf{coarsely surjective}, i.e.~there is a uniformly bounded family $\mathcal{U}$ in $Y$ such that $Y \subseteq \st(f(X), \mathcal{U})$;
\item $f$ is a a \textbf{coarse embedding}, i.e.~ for every uniformly bounded family $\mathcal{U}$ in $Y$, $f^{-1}(\mathcal{U}) = \{f^{-1}(U) \mid U \in \mathcal{U} \}$ is uniformly bounded in $X$.
\end{itemize}

\begin{Proposition} \label{disc1}
 If $X$ is a hybrid large scale space, then it contains a topologically discrete subset $Y$ such that the inclusion $i:Y\to X$, with the induced ls-structure on $Y$, is a large scale equivalence. If $X$ is proper, then $Y$ can be chosen such that the bounded subsets of $Y$ are finite.
\end{Proposition}
\begin{proof}
Pick an open scale $\mathcal{U}$ of $X$. Let $Y$ be a maximal subset of $X$ with respect to the following property: $Y\cap U$ contains at most one point for all $U\in \mathcal{U}$. The inclusion $i: Y \rightarrow X$ is clearly coarsely surjective, and it is a coarse embedding since the ls-structure on $Y$ is induced by $X$.
\end{proof}

Proposition \ref{disc1} is well known in the case where $X$ is a metric space. Indeed, the discrete subset $Y$ can be realised as a maximal $1$-separated subset of $X$, that is, a maximal subset $Y$ with respect to the property that any two points of $Y$ are at least distance $1$ apart (see for example Section 1 of \cite{Gromov93}).

\section{Neighbourhood operators}
In order to make clear the connection between the results for (hybrid) large scale spaces contained in this paper and the classical topological results for topological spaces, we prove some results in the context of a set equipped with a neighbourhood operator satisfying certain axioms. By a \textbf{neighbourhood operator} on a set $X$ we mean a binary relation $\prec$ on the power set $\mathcal{P}(X)$ of $X$ such that $A \prec B \implies A \subseteq B$. If $A \prec B$, we say that $B$ is a neighbourhood of $A$ with respect to $\prec$. Neighbourhood operators appear in many places in the literature: see for example \cite{Czaszar} for applications to topology, or \cite{HolgateSlapal} for a more categorical approach. For our purposes, we will be interested in neighbourhood operators $\prec$ on a set $X$ satisfying the following conditions:

\begin{itemize}
\item[$\mathsf{(N0)}$] $A \prec X$ for all $A \subseteq X$.
\item[$\mathsf{(N1)}$] if $A \prec B$ then $X \setminus B \prec X \setminus A$.
\item[$\mathsf{(N2)}$] if $A \prec B \subseteq C$, then $A \prec C$.
\item[$\mathsf{(N3)}$] if $A \prec N$ and $A' \prec N'$ then $A \cup A' \prec N \cup N'$.
\end{itemize}
It is easy to see that, together, axioms $\mathsf{(N0)}-\mathsf{(N3)}$ imply:

\begin{itemize}
\item[$\mathsf{(N0')}$] $\varnothing \prec A$ for all $A \subseteq X$.
\item[$\mathsf{(N2')}$] if $A \subseteq B \prec C$ then $A \prec C$.
\item[$\mathsf{(N3')}$] if $A \prec N$ and $A' \prec N'$ then $A \cap A' \prec N \cap N'$.
\end{itemize}
We now introduce some examples of neighbourhood operators, the first three of which are the most important for our purposes.

\begin{itemize}
\item the \textbf{topological neighbourhood operator} on a topological space $X$: define $A \prec B$ if and only if $B$ contains an open set containing $\mathsf{cl}(A)$.
\item the \textbf{coarse neighbourhood operator} on an ls-space $X$: define $A \prec B$ if and only if $B$ is a \textbf{coarse neighbourhood} of $A$, that is, $A \subseteq B$ and for every uniformly bounded cover $\mathcal{U}$ of $X$, $\st(A, \mathcal{U})$ is contained in $B \cup K$ for some weakly bounded set $K$.
\item the \textbf{hybrid neighbourhood operator} on an hls-space $X$: define $A \prec B$ if and only if $B$ is a neighbourhood of $A$ with respect to the topological neighbourhood operator and the coarse neighbourhood operator on $X$. 
\item the \textbf{uniform neighbourhood operator} on a uniform space $X$ (see for example \cite{Isbell}): define $A \prec B$ if there is a uniform cover $\mathcal{U}$ such that such that $\st(A, \mathcal{U}) \subseteq B$. 
\end{itemize}

If $B$ is a neighbourhood of $A$ with respect to the topological neighbourhood operator, then we say that $B$ is a topological neighbourhood of $A$, and similarly for the coarse and hybrid neighbourhood operators. For proper metric spaces the notion of coarse neighbourhood is closely related to the notion of asymptotic neighbourhood in \cite{BellDranAsym}. Indeed, a coarse neighbourhood of a subset $A$ of a proper metric space is nothing but an asymptotic neighbourhood of $A$ which contains $A$. Note that what we in this paper call the topological neighbourhood operator does not capture the neighbourhood relation in the usual sense (that is, where $B$ is a neighbourhood of $A$ if and only if $B$ contains an open set which contains $A$). Indeed, the usual neighbourhood relation does not satisfy $\mathsf{(N1)}$, while one can check that the topological neighbourhood operator above does. 

\begin{Observation}
Conditions $\mathsf{(N0)}$-$\mathsf{(N3)}$ are satisfied by all four examples given above.
\end{Observation}

\begin{Remark}
Clearly every neighbourhood operator $\prec$ on a set $X$ which satisfies $\mathsf{(N0)}$ and $\mathsf{(N3')}$ induces a topology on $X$ wherein a set $U \subseteq X$ is open if and only if for every $x \in U$, $\{x\} \prec U$. For a $T_1$ topological space and the topological neighbourhood operator, this recovers the original topology. For a uniform space and the uniform neighbourhood operator, this recovers the topology induced by the uniform structure in the usual sense. For a large scale space and the coarse neighbourhood operator, every superset of a singleton set is a coarse neighbourhood of that set, so the induced topology is the discrete topology.
\end{Remark}

\begin{Definition}
Let $X$ and $Y$ be sets equipped with neighbourhood operators $\prec_X$ and $\prec_Y$ respectively. A set map $f: X \rightarrow Y$ is called \textbf{neighbourhood continuous with respect to $\prec_X$ and $\prec_Y$} if $A \prec_Y B \implies f^{-1}(A) \prec_X f^{-1}(B)$ for any subsets $A$ and $B$ of $Y$.
\end{Definition}

We now show that neighbourhood continuity generalises both topological continuity and being slowly oscillating for maps into $[0,1]$. 

\begin{Proposition}\label{neighcont}
Let $X$ and $Y$ be topological spaces and let $f: X \to Y$ be a set map. If $f$ is topologically continuous, then it is neighbourhood continuous with respect to the topological neighbourhood operators on $X$ and $Y$. If $Y$ is a $T_1$-space then the converse also holds.
\end{Proposition}
\begin{proof}
It is easy to check that if $f$ is topologically continuous then it is also neighbourhood continuous. Suppose then that $f$ is neighbourhood continuous, $Y$ is a $T_1$ space and let $A$ be an open set with $f(x) \in A$. Since the point $f(x)$ is closed, we have $\{f(x)\} \prec A$. Thus we have $f^{-1}(f(x)) \prec f^{-1}(A)$, which gives us continuity at $x$. 
\end{proof}

For convenience, when we are referring to a map $f$ from a set $X$ equipped with a neighbourhood operator $\prec$ to a subset of $\mathbb{R}$, we say that $f$ is neighbourhood continuous to mean that it is neighbourhood continuous with respect to the neighbourhood operator $\prec$ and the topological neighbourhood operator on the codomain.

\begin{Lemma}\label{easyintervals}
Let $X$ be a set and $\prec$ a neighbourhood operator on $X$ satisfying $\mathsf{(N0)}-\mathsf{(N3)}$. A map $f: X \rightarrow [0,1]$ is neighbourhood continuous if and only if for every $a < b$ in $[0,1]$ we have $f^{-1}([0,a]) \prec f^{-1}([0, b))$.
\end{Lemma}
\begin{proof}
$(\Rightarrow)$ is obvious.

$(\Leftarrow)$ Suppose $A$ has neighbourhood $N$ in $[0,1]$ relative to the topological neighbourhood operator. We must show  that $f^{-1}(A) \prec f^{-1}(N)$. We may suppose that $N$ is open and $A$ is closed since the neighbourhood operator $\prec$ satisfies $\mathsf{(N2)}$ and $\mathsf{(N2')}$. Since $[0,1]$ is compact, there is an $\varepsilon > 0$ such that $B(A, \varepsilon) \cap [0,1] \subseteq N$ (indeed, the function $x \mapsto d(x, X \setminus N)$ achieves a minimum on $A$ which cannot be $0$). The connected components of $A' = \{x \in [0,1] \mid \exists_{a \in A}\ d(a, x) \leq \varepsilon/2\}$ have diameter at least $\varepsilon$, so $A'$ is a finite union of closed intervals. Moreover, $A'$ contains $A$ and is contained in $N' = B(A', \varepsilon/2) \cap [0,1]$, which is in turn contained in $N$. Using $\mathsf{(N2)}$ and $\mathsf{(N2')}$ again together with $\mathsf{(N3)}$, we may thus reduce to the case where $A = [a,b]$ and $N = (a', b') \cap [0,1]$ for $a' < a$ and $b < b'$, and we can choose $b' < 1$ if $b < 1$. If $0 < a' < b' < 1$, then noticing that $[a,b] = [0, b] \cap ([0,1] \setminus [0, a) )$ and $(a',b') = [0, b') \cap ([0,1] \setminus [0, a'])$ and using condition $\mathsf{(N2')}$ and $\mathsf{(N1)}$, we can reduce to the case of $A = [0, x]$ and $N = [0, y)$ for $x < y$. If $a = 0$ and $b < 1$, then we are already reduced to the desired case. Finally, if $A = [0,1]$, then by $\mathsf{(N0)}$, we have that $f^{-1}(A) = X \prec X = f^{-1}(N)$, so we can discard this case.
\end{proof}

\begin{Proposition}\label{neighslow}
Let $X$ be an ls-space and $f: X \rightarrow [0,1]$ a set map. Then $f$ is slowly oscillating if and only if $f$ is neighbourhood continuous with respect to the coarse neighbourhood operator on $X$ and the topological neighbourhood operator on $[0,1]$.
\end{Proposition}
\begin{proof}
It is enough to consider the case when $X$ is coarsely connected. 

$(\Rightarrow)$ Let $a < b$, with $b-a = \varepsilon$. If $\mathcal{U}$ is a uniformly bounded cover, then there is a bounded set $K$ in $X$ such that $f(U)$ has diameter less than $\varepsilon/2$ for every $U$ in $\mathcal{U}$ not contained in $K$. Thus $\st(f^{-1}([0,a]), \mathcal{U})$ is contained in $f^{-1}([0, a+\varepsilon)) \cup K$, which gives the result by Lemma \ref{easyintervals}. 

$(\Leftarrow)$ Suppose that $f$ is not slowly oscillating. Then there is an $\varepsilon > 0$ and a uniformly bounded cover $\mathcal{U}$ of $X$ such that $Y = \bigcup \{ U \in \mathcal{U} \mid \mathsf{diam}(f(U)) > \varepsilon \}$ is unbounded. Divide $[0,1]$ into consecutive closed intervals $I_1, \ldots, I_k$ of length less than $\varepsilon/2$ with non-empty interior, and let $I_0 = I_{k+1} = \varnothing$ for convenience. Then there exists a $1 \leq m \leq k$ such that $f^{-1}(I_m) \cap Y$ is unbounded (otherwise $Y$ is a finite union of bounded sets). The subset $N = I_{m-1} \cup I_m \cup I_{m+1}$ is a topological neighbourhood of $I_m$, but $\st(f^{-1}(I_m), \mathcal{U}) \setminus f^{-1}(N)$ is not bounded, so $f^{-1}(N)$ is not a coarse neighbourhood of $f^{-1}(I_m)$.
\end{proof}

\begin{Proposition}
Let $X$ be an hls-space and $f: X \rightarrow [0,1]$ a set map. Then $f$ is continuous and slowly oscillating if and only if $f$ is neighbourhood continuous with respect to the hybrid neighbourhood operator on $X$ and the topological neighbourhood operator on $[0,1]$.
\end{Proposition}
\begin{proof}
This follows from Proposition \ref{neighcont} and \ref{neighslow} above.
\end{proof}

Now that we have motivated the notion of neighbourhood continuity, we are ready to prove some general results about neighbourhood operators. Before we do, we introduce a ``normality'' condition on a neighbourhood operator $\prec$.

\begin{itemize}
\item[$\mathsf{(N4)}$] for every pair of subsets $A \prec C$, there is a subset $B$ with $A \prec B \prec C$.
\end{itemize}

\begin{Lemma}\label{neighdense}
Let $X$ be a set equipped with a neighbourhood operator $\prec$ satisfying $\mathsf{(N0)}$--$\mathsf{(N3)}$ and let $\{A_s\}_{s\in S}$ be a family of subsets of $X$ indexed by a dense subset $S$ of $[0,1]$. If, for each $s < t\in S$, we have $A_s \prec A_t$, then the function $f:X\to [0,1]$ defined by
 \begin{align*}
 f(x) = & \inf\{t \mid x\in A_t\} 
 \end{align*}
 is neighbourhood continuous.
\end{Lemma}
\begin{proof}
Let $[0, a] \subseteq [0, b)$ be subsets of $[0,1]$. Pick $s, s' \in S$ such that $a < s < s'< b$. Then
$$
f^{-1}([0,a]) \subseteq A_s \prec A_{s'} \subseteq f^{-1}([0, b)).
$$
Thus by Lemma \ref{easyintervals} we obtain the result.
\end{proof}

\begin{Theorem}[Urysohn's Lemma for neighbourhood operators] \label{neighurysohn}
Let $X$ be a set and $\prec$ a neighbourhood operator satisfying $\mathsf{(N0)}$--$\mathsf{(N3)}$. Then the following are equivalent:
\begin{itemize}
\item[(1)] $\prec$ satisfies $\mathsf{(N4)}$,
\item[(2)] for any subsets $A$ and $B$ of $X$ such that $A \prec X \setminus B$, there is a neighbourhood continuous function $f: X \rightarrow [0,1]$ such that $f(A) \subseteq \{0\}$ and $f(B) \subseteq \{1\}$.
\end{itemize} 
\end{Theorem}
\begin{proof}
(1) $\implies$ (2): By Lemma \ref{neighdense}, it is enough to produce a family of subsets $A_s$ indexed by a dense subset $S$ of $[0,1]$ such that $A_s \prec A_t$ whenever $s < t$. Using $\mathsf{(N4)}$ we can define such subsets indexed by the dyadic fractions, starting with $A_0 = A$ and $A_1 = X \setminus B$.

(2) $\implies$ (1): Given $A \prec N$, construct a neighbourhood continuous map $f$ taking $A$ to $0$ and $X \setminus N$ to $1$. Then $f^{-1}([0, 1/2))$ is the required intermediate neighbourhood.
\end{proof}

Notice that the proof of Theorem \ref{neighurysohn} is a straightforward adaptation of the standard proof of Urysohn's Lemma from topology. We can recover the classical Urysohn's Lemma from Theorem \ref{neighurysohn} above.

\begin{Lemma}
Let $X$ be a topological space. Then $X$ is normal if and only if the topological neighbourhood operator on $X$ satisfies $\mathsf{(N4)}$.
\end{Lemma}

\begin{Corollary}[Urysohn's Lemma]
Let $X$ be a normal topological space. Then for any closed disjoint subsets $A$ and $B$ of $X$ there is a continuous map $f: X \rightarrow [0,1]$ such that $f(A) \subseteq \{0\}$ and $f(B) \subseteq \{1\}$. 
\end{Corollary}

\section{Hybrid large scale Urysohn's Lemma}
In this section we apply the results of the previous section to prove results for hybrid large scale spaces. 

\begin{Definition}
Let $X$ be an ls-space and $A$, $B$ be subsets of $X$. We say that $A$ and $B$ are \textbf{coarsely separated} for every uniformly bounded family $\mathcal{U}$ in $X$, $\st(A, \mathcal{U}) \cap \st(B, \mathcal{U})$ is weakly bounded.  
\end{Definition}

Note that in the case of metric spaces, this is the same as saying that $A$ and $B$ \textbf{diverge} in the sense of \cite{Dranetal}. Clearly if $A$ and $B$ are disjoint subsets of an ls-space $X$, then $A$ and $B$ are coarsely separated if and only if $X \setminus B$ is a coarse neighbourhood of $A$. 

\begin{Definition}
Let $X$ be a hybrid large scale space.  We say that $X$ is \textbf{hybrid large scale normal} (or hls-normal) if for every closed subset $A$ and every hybrid neighbourhood $N$ of $A$, there is a closed subset $V$ of $X$ such that $V$ is a hybrid neighbourhood of $A$ and $N$ is a hybrid neighbourhood of $V$. We say that an ls-space is \textbf{ls-normal} if it is hybrid large scale normal when equipped with the discrete topology.
\end{Definition}

\begin{Lemma}\label{neighN4hyb}
An hls-space is hls-normal if and only if the hybrid neighbourhood operator satisfies $\mathsf{(N4)}$. 
\end{Lemma}
\begin{proof}
$(\Rightarrow)$ Suppose $A$ has a hybrid neighbourhood $N$. In particular then, $\mathsf{cl}(A)$ is contained in the interior of $N$. But since $\mathsf{cl}(A) \subseteq \st(A, \mathcal{U})$ for any open scale $\mathcal{U}$, $N$ is a coarse neighbourhood of $\mathsf{cl}(A)$. Thus $N$ is a hybrid neighbourhood of $\mathsf{cl}(A)$, and we can obtain an intermediate hybrid neighbourhood as required. 

$(\Leftarrow)$ Consider a closed subset $A$ and a hybrid neighbourhood $N$ of $A$. Condition $\mathsf{(N4)}$ gives us an intermediate hybrid neighbourhood $V$. Taking $\mathsf{cl}(V)$, by similar arguments to the previous direction, produces the required closed intermediate hybrid neighbourhood.
\end{proof}

Combining Lemma \ref{neighN4hyb} and Lemma \ref{neighurysohn} we obtain:

\begin{Corollary}[Urysohn's Lemma for hybrid large scale spaces]\label{UrysohnLemmaInHLS}
Let $X$ be an hls-space. Then the following are equivalent:
\begin{itemize}
\item[(1)] $X$ is hls-normal,
\item[(2)] if $A$ has hybrid neighbourhood $N$, then there is a continuous slowly oscillating map $f: X \rightarrow [0,1]$ such that $f(A) \subseteq \{0\}$ and $f(X \setminus N) \subseteq \{1\}$,
\item[(3)] for any closed coarsely separated disjoint subsets $A$ and $B$ of $X$ there is a continuous slowly oscillating map $f: X \rightarrow [0,1]$ such that $f(A) \subseteq \{0\}$ and $f(B) \subseteq \{1\}$.
\end{itemize} 
\end{Corollary}
\begin{proof}
The only part which needs proving is the equivalence of (2) and (3). Clearly (2) implies (3). To show (3) implies (2), notice that if $A$ has hybrid neighbourhood $N$, then $\mathsf{cl}(A)$ and $\mathsf{cl}(X \setminus N)$ are closed, coarsely separated and disjoint.
\end{proof}

\section{Hls-normal spaces}
In this section we look at some more properties of hls-normal spaces, as well as some classes of examples of hls-normal hls-spaces. 

\begin{Lemma}\label{NormalityIssue}
 The topology of any hls-normal hls-space is normal.
\end{Lemma}
\begin{proof}
It is enough to consider the case where $X$ is $\mathcal{U}$-connected for some open scale $\mathcal{U}$ (since the $\mathcal{U}$-components are open-closed). Notice that the topology induced on any closed and bounded subset $Y$ of $X$ is normal due to the fact that any two disjoint, closed subsets of $Y$ are coarsely separated in $X$. Express $X$ as in Lemma \ref{DecompositionLemma}. Since each $A_i$ is topologically normal, so is their directed union. Indeed, for any closed subsets $A$ and $B$ of $X$, define a continuous function $f_1$ taking $A \cap A_1$ to $0$ and $B \cap A_1$ to $1$. Then use the Tietze Extension Theorem to extend the function which agrees with $f_1$ on $A_1$ and which sends $A \cap A_2$ to $0$ and $B \cap A_2$ to $1$ to all of $A_2$. Continuing this process one defines a function $f$ inductively which sends $A$ to $0$ and $B$ to $1$, and which is continuous by the conditions on the $A_n$.
\end{proof}

 \begin{Theorem}\label{BigNormalityIssue}
 If $X$ is a hybrid large scale space, then the following conditions are equivalent:
 \begin{itemize}
 \item[(1)] $X$ is hls-normal,
 \item[(2)] $X$ is ls-normal as an ls-space and the topology of $X$ is normal.
 \end{itemize}
\end{Theorem}
\begin{proof}
 $ (1) \Rightarrow (2):$ In view of Lemma \ref{NormalityIssue}, the topology of $X$ is normal.
 Suppose $A$ and $B$ are two disjoint and coarsely separated subsets of $X$. Then the closures of $A$ and $B$ are also coarsely separated and disjoint outside of a bounded set $K$ by Lemma \ref{FirstLemma}. Thus there is a a slowly oscillating continuous function $f:X\to [0,1]$
sending $cl(A)\setminus K$ to $0$ and sending $cl(B) \setminus K$ to $1$ by hls-normality.
Redefine $f$ on $K$ such that it sends $A\cap K$ to $0$ and $B\cap K$ to $1$, and notice that the new $f$ is slowly oscillating.

$ (2) \Rightarrow (1):$  Suppose $A$ 
is a closed subset of $X$ and $U$ is its hybrid neighbourhood. We need to find a closed coarse neighbourhood $V$ of $A$ such that $U$ is a coarse neighbourhood of $V$. Pick an open scale $\mathcal{U}$ of $X$ and pick a coarse neighbourhood $W$ of $A$
so that $U$ is a coarse neighbourhood of $W$.
The set $V = \mathsf{cl}(\st(W,\mathcal{U}))$ is closed and is a coarse neighbourhood of $A$. However, it may not be contained in
$U$. Nonetheless, since $V$ is contained in $\st(\st(W, \mathcal{U}), \mathcal{U})$, it is contained in $U \cup K$ for some bounded set $K$. Using the topological normality of $X$, we may find a closed topological neighbourhood $V'$ of $A$ which is contained in $U$. The required intermediate closed hybrid neighbourhood between $A$ and $U$ is then $(V \setminus \st(K, \mathcal{U})) \cup (V' \cap \mathsf{cl}(\st(K, \mathcal{U})))$. 
\end{proof}

Thus we may say that
$$
\textbf{hls-normality}\ = \textbf{topological normality}\ +\ \textbf{ls-normality}.
$$
We should note that the compatibility axiom played a crucial role in the proof of this fact. We now present some examples of hls-normal spaces. In particular, we show that metric spaces, both with the usual ls-structure and with the $C_0$ structure introduced by Wright, are hls-normal, as is any set equipped with the maximal uniformly locally finite ls-structure.

\begin{Definition}[Wright \cite{Wright}]
Let $(X,d)$ be a metric space. Let $\mathcal{L}$ be the collection of all families $\mathcal{U}$ of subsets of $X$ such that for every $\varepsilon > 0$, there is a bounded set $B \subseteq X$ such that $\mathsf{diam}(U) \leq \varepsilon$ for all $U \in \mathcal{U}$ not intersecting $B$. Then $\mathcal{L}$ is an ls-structure, called the $C_0$ \textbf{ls-structure} associated to the metric $d$.
\end{Definition}

\begin{Proposition}\label{C0metricnormality}
Let $X$ be a metric space equipped with the metric topology. Then $X$ equipped with either the metric or $C_0$ ls-structure is hls-normal. 
\end{Proposition}
\begin{proof}
The same construction works for both the metric and $C_0$ ls-structures.
Let $A$ be a closed subset  and $U$ a hybrid neighbourhood of $A$. Let $V$ be the set of all points $x \in X$ such that $d(x, A) \leq d(x, X \setminus U)$. Clearly $V$ is closed and contains a neighbourhood of $A$. We claim that $V$ is an intermediate coarse neighbourhood between $A$ and $U$. Indeed, let $\mathcal{U}$ be a cover of $X$ by balls of bounded radii. If $\st(A, \mathcal{U})$ intersects $X \setminus V$ in an unbounded set, then it is easy to check that $\st(A, \mathcal{U}')$ intersects $X \setminus U$ in an unbounded set, where $\mathcal{U}'$ is the set formed from $\mathcal{U}$ by replacing every ball $B(x, R)$ by $B(x, 2R)$. This is a contradiction since $\mathcal{U}'$ is uniformly bounded whenever $\mathcal{U}$ is for both ls-structures. A similar argument shows that $U$ is a hybrid neighbourhood of $V$.
\end{proof}

A family $\mathcal{U}$ of subsets of a set $X$ is \textbf{uniformly locally finite}
if there is a natural number $m$ so that $\mathrm{card}(st(x,\mathcal{U}))\leq m$ for all $x\in X$

\begin{Definition}[Sako \cite{Sako}]
 A large scale space $X$ is \textbf{uniformly locally finite} if every uniformly bounded cover $\mathcal{U}$ of $X$ is uniformly locally finite.
\end{Definition}

On any set $X$ the collection of all uniformly locally finite families forms an ls-structure (called the \textbf{maximal uniformly locally finite ls-structure}), which is the largest uniformly locally finite ls-structure on $X$. Viewed as a coarse structure in the sense of Roe, the maximal uniformly locally finite structure is nothing but the universal bounded geometry structure in the sense of \cite{Roe}.

\begin{Proposition}
Let $X$ be a set equipped with the maximal uniformly locally finite ls-structure. Then $X$ is an ls-normal space.
\end{Proposition}
\begin{proof}
Note that given any two coarsely separated subsets $A$ and $B$ of $X$ relative to this structure, one of them is finite. The result follows easily from this observation. 
\end{proof}

\section{Non-normal spaces}
At this point, one might ask if there are any hls-spaces which are not hls-normal. An example of an ls-space which is not ls-normal is described below in Proposition \ref{Nonnormal}. In Section \ref{secgroups} we will also see a class of topological groups which are not hls-normal as hls-spaces.

\begin{Proposition} \label{Nonnormal}
Let $X$ be the subset of the upper half-plane of $\mathbb{Z}^2$ given by $-y \leq x \leq y$, $y > 0$. Let $A = \{(x, x) \mid x\in \mathbb{Z}, x > 0\}$ and $B = \{(-x, x) \mid x \in \mathbb{Z}, x > 0\}$. Define an ls-structure on this space as follows: let $\mathscr{L}$ be the set of all uniformly locally finite families $\mathcal{V}$ such that for any scale $\mathcal{U}$ in the metric ls-structure on $X$ the set
of $V\in \mathcal{V}$ intersecting $\st(A\cup B,\mathcal{U})$ is uniformly bounded in the metric ls-structure on $X$. Then:
\begin{itemize}
\item[(1)] The collection $\mathscr{L}$ is a uniformly locally finite ls-structure on $X$. The uniformly bounded families with respect to the metric ls-structure are members of $\mathcal{L}$.
\item[(2)] $A$ and $B$ are coarsely separated in $(X, \mathscr{L})$.
\item[(3)] there is no slowly oscillating (with respect to $\mathscr{L})$ function $f: X \rightarrow [0,1]$ such that $f(A) = 0$, $f(B) = 1$.
\end{itemize}
so that, in particular, $X$ is not hls-normal.
\end{Proposition}
\begin{proof}
(1) Suppose $\mathcal{V}_1, \mathcal{V}_2 \in \mathscr{L}$ are covers. We will show that $\st(\mathcal{V}_1, \mathcal{V}_2)$ is in $\mathscr{L}$. Clearly $\st(\mathcal{V}_1, \mathcal{V}_2)$ is uniformly locally finite.  Suppose then that $\mathcal{U}$ is a scale in the metric ls-structure on $X$, and let $U = \st(A \cup B, \mathcal{U})$. Let $\mathcal{V}_2'$ be the family of elements of $\mathcal{V}_2$ intersecting $U$, $\mathcal{V}_1'$ the family of elements of $\mathcal{V}_1$ intersecting $\st(U, \mathcal{V}_2')$ and $\mathcal{V}_2''$ the family of elements of $\mathcal{V}_2$ intersecting $\st(\st(U, \mathcal{V}_2'), \mathcal{V}_1')$. Each of these families is uniformly bounded in the metric ls-structure, and the family of elements of $\st(\mathcal{V}_1, \mathcal{V}_2)$ intersecting $U$ clearly refines $\st(\mathcal{V}_1', \mathcal{V}_2' \cup \mathcal{V}_2'')$, so it is uniformly bounded in the metric ls-structure.

(2) Suppose $\mathcal{V}\in \mathscr{L}$.
The set $\st(A,\mathcal{V})\cap \st(B,\mathcal{V})$ is contained in the union of  all sets $V_1\cap V_2$, where $V_1\in \mathcal{V}$ intersects $A$ and $V_2\in \mathcal{V}$ intersects $B$, so the family of those sets forms a uniformly bounded family $\mathcal{U}$ in the metric ls-structure on $X$.
Therefore $\st(A,\mathcal{V})\cap \st(B,\mathcal{V})\subset \st(A,\mathcal{U})\cap \st(B,\mathcal{U})$ which is finite.

(3) Suppose such an $f$ exists. Let $X_i$ be the set $\{(x,i) \mid x \in \mathbb{Z}\} \cap X$. Consider the cover of $X$
by $2$-balls in the $l_1$-metric structure on $X$.
Since $f$ is in particular slowly oscillating with respect to the metric ls-structure, there is some $M > 0$ such that outside
of the $M$-ball at $(1,1)$ one has $|f(z_1)-f(z_2)| < 1/6$ if $z_1$, $z_2$ are on the same horizontal line and their distance
is $1$. Therefore
$f^{-1}([1/6, 1/3])$ and $f^{-1}([2/3, 5/6])$ both intersect $X_i$ for $i > M+2$. Take 
$z_i\in f^{-1}([1/6, 1/3])\cap X_i$ and $w_i\in f^{-1}([2/3, 5/6])\cap X_i$ for $i > M+2$.
Notice $\mathcal{Z} = \{z_i,w_i\}_{i > M+2}$ belongs to $\mathscr{L}$. Indeed,
since $f$ is slowly oscillating with respect to the metric structure, the union of $\mathcal{Z}$ must be coarsely separated from $A$ and $B$
in the metric ls-structure. Thus the family $\mathcal{Z}$ is an element of $\mathscr{L}$. However, $|f(z_i)-f(w_i)| \ge 1/3$ for all $i> M+2$, which contradicts the fact that $f$ is slowly oscillating with respect to $\mathscr{L}$.
\end{proof}

\section{The Tietze Extension Theorem}
As with Urysohn's Lemma, the proof of the Tietze Extension Theorem for neighbourhood operators is a straightforward adaptation of the classical proof, and gives us the result for (hybrid) large scale spaces as a corollary.

\begin{Lemma}\label{addition}
Let $\prec$ be a neighbourhood operator on a set $X$ satisfying $\mathsf{(N0)}-\mathsf{(N3)}$ and let $f,g: X \rightarrow [-M, M]$ be two neighbourhood continuous maps. Then $f+g$ is neighbourhood continuous.
\end{Lemma}
\begin{proof}
By Lemma \ref{easyintervals} (since $[-2M, 2M]$ is homeomorphic to $[0,1]$), it is enough to show that for any interval $[-2M,b]$ in $\mathbb{R}$ and $\varepsilon >0$, that $(f+g)^{-1}([-2M,b])$ has neighbourhood $(f+g)^{-1}([-2M, b+\varepsilon))$ relative to $\prec$. Cover $[-M, M]$ by finitely many intervals $I_n = [-M,\ n\varepsilon/4 + \varepsilon/4]$ and $J_n = [-M,\ b - n\varepsilon/4]$, $n \in \mathbb{Z}$. It follows that
\begin{align*}
(f+g)^{-1}([-2M,b]) & \subseteq \bigcup_n f^{-1}(I_n) \cap g^{-1}(J_n) \\
& \subseteq \bigcup_n f^{-1}(B(I_n, \varepsilon/4)) \cap g^{-1}(B(J_n, \varepsilon/4)) \\
& \subseteq (f+g)^{-1}([-2M, b+\varepsilon)).
\end{align*}
Since $f$ and $g$ are neighbourhood continuous and $\prec$ satisfies $\mathsf{(N0)}-\mathsf{(N3)}$, we have
$$
\bigcup_n f^{-1}(I_n) \cap g^{-1}(J_n) 
\prec \bigcup_n f^{-1}(B(I_n, \varepsilon/4)) \cap g^{-1}(B(J_n, \varepsilon/4)) 
$$
which completes the proof.
\end{proof}

\begin{Lemma}\label{neighlemmaforTietze}
Let $\prec$ be a neighbourhood operator on a set $X$ satisfying $\mathsf{(N0)}-\mathsf{(N3)}$. Suppose $g_n: X \rightarrow [-m_n, m_n]$ is a sequence of neighbourhood continuous maps such that 
$$
\sum\limits_{i=1}^{\infty} m_n = m < \infty.
$$
Then $f = \sum_{i=1}^{\infty} g_n : X \rightarrow \mathbb{R}$ is neighbourhood continuous.
\end{Lemma}
\begin{proof}
By Lemma \ref{easyintervals} (since $[-m, m]$ is homeomorphic to $[0,1]$), it is enough to show that for any interval $[-m,b]$ in $\mathbb{R}$ and $\varepsilon >0$, that $f^{-1}([-m,b])$ has neighbourhood $f^{-1}([-m, b+\varepsilon))$ relative to $\prec$. Pick $M$ such that $\sum_{n=M}^{\infty} m_n < \varepsilon / 4$ and let $f' = \sum_{n=1}^{M-1} g_n$. Then $f'$ is neighbourhood continuous by Lemma \ref{addition}, so $f'^{-1}([-m,b + \varepsilon/4])$ has neighbourhood $f'^{-1}([-m, b+\varepsilon/2))$. But 
$$
f^{-1}([-m,b]) \subseteq f'^{-1}([-m,b + \varepsilon/4])
$$
and
$$
f'^{-1}([-m, b+\varepsilon/2)) \subseteq f^{-1}([-m, b+\varepsilon))
$$
from which we obtain the required result.
\end{proof}

\begin{Definition}
Let $X$ be a set and $\prec$ a neighbourhood operator. If $A$ is a subset of $X$, then the \textbf{induced neighbourhood operator} $\prec_A$ on subsets of $A$ is defined as follows: $S \prec_A T$ precisely when there exists a subset $T'$ of $X$ such that $S \prec T'$ as subsets of $X$ and $T = T' \cap A$. 
\end{Definition}

\begin{Observation}
If a neighbourhood operator $\prec$ on $X$ satisfies $\mathsf{(N0)}-\mathsf{(N4)}$, then so does the induced neighbourhood operator on any subset.
\end{Observation}

\begin{Theorem}[Tietze Extension Theorem for neighbourhood operators]\label{neighTietze}
Let $X$ be a set and $\prec$ a neighbourhood operator satisfying $\mathsf{(N0)}-\mathsf{(N3)}$. Then $\prec$ satisfies $\mathsf{(N4)}$ if and only if for any function $f: A \rightarrow [-2, 2]$ from a subset $A$ of $X$ which is neighbourhood continuous with respect to the operator induced by $\prec$ on $A$ and the topological neighbourhood operator on $[-2,2]$, there is a neighbourhood continuous function $g: X \rightarrow [-2,2]$ which extends $f$. 
\end{Theorem}
\begin{proof}
The proof follows the classical topological proof closely. Suppose $\prec$ satisfies $\mathsf{(N4)}$.  

\textbf{Claim:} Given a neighbourhood continuous map $f: A \rightarrow [-3m, 3m]$, $m > 0$, there is a neighbourhood continuous map $g: X \rightarrow  [-m, m]$ such that $|f(x) - g(x)| \leq 2m$ for all $a \in A$. 

\textbf{Proof of Claim:} Let $S = f^{-1}([-3m, -m])$ and $T = f^{-1}([m, 3m])$. Since $f$ is neighbourhood continuous, we have $S \prec_A A \setminus T$. It follows from the definition of $\prec_A$ and condition $\mathsf{(N2)}$ that $S \prec X \setminus T$, so by Theorem \ref{neighurysohn}, there is a neighbourhood continuous map $g': X \rightarrow [0, 1]$ such that $g'(S) \subseteq \{0\}$ and $g'(T) \subseteq \{1\}$. Composing with the appropriate linear map $[0,1] \rightarrow [-m, m]$, we obtain the required map $g$. This proves the claim.

Now, define $m(n) = 2^{n+1}/3^n$ for $n \geq 0$. Using the Claim, inductively construct  a sequence of functions $g_n: X \rightarrow [-m(n), m(n)]$ which are neighbourhood continuous and such that for all $a \in A$, 
$$
|f(a) - \sum\limits_{i=1}^{n+1} g_i(a)| \leq 2m(n).
$$  
Then the map $g = \sum_{i=1}^\infty g_i$ is neighbourhood continuous by Lemma \ref{neighlemmaforTietze} and agrees with $f$ on $A$. 

For the other direction, note that if $A \prec N$, then the function which sends $A$ to $0$ and $X \setminus N$ to $1$ is neighbourhood continuous on $A \cup (X \setminus N)$. Thus by Theorem \ref{neighurysohn} we have the result.

\end{proof}

\begin{Corollary}[Tietze Extension Theorem]
Let $X$ be a normal topological space and let $A$ be a closed subset of $X$. Then any continuous function $f: A \rightarrow [0,1]$ extends to a continuous function $g: X \rightarrow [0,1]$. 
\end{Corollary}
\begin{proof}
In order to apply Theorem \ref{neighTietze}, we have only to show that the function $f$ is continuous with respect to the neighbourhood operator $\prec_A$ induced by the topological neighbourhood operator on $X$. Suppose $S$ and $T$ are subsets of $A$, and that the closure of $S$ in $A$ is contained in a subset $V \subseteq T$ which is open in $A$. Since $A$ is closed, the closure of $S$ in $A$ coincides with the closure of $S$ in $X$. Let $V'$ be an open set in $X$ such that $V' \cap A = V$. Thus the closure of $S$ (in $X$) is contained in $V'$ which is contained in $T \cup X \setminus A$, so $S \prec_A T$. This shows that the topological neighbourhood operator associated to the subspace topology is contained (as a relation) in $\prec_A$, which gives the required result.
\end{proof}

We also obtain a result for hybrid large scale spaces. Note that if $A$ is a subset of an ls-space $X$, then the coarse neighbourhood operator induced by the subspace ls-structure on $A$ coincides with the neighbourhood operator induced on $A$ by the coarse neighbourhood operator on $X$. 

\begin{Corollary}[Tietze Extension Theorem for hybrid large scale spaces]\label{TietzeTheorem}
Let $X$ be an hls-space. Then $X$ is hls-normal if and only if for any closed subset $A$ of $X$, any continuous slowly oscillating function $f: A \rightarrow [0,1]$ extends to a continuous slowly oscillating function $g: X \rightarrow [0,1]$.
\end{Corollary}

\begin{Corollary}
 Given a metric space $X$, any bounded continuous slowly oscillating function on a closed subset of $X$ to $\mathbb{R}$ extends to a bounded continuous slowly oscillating function on the whole of $X$ to $\mathbb{R}$.
 \end{Corollary}
 \begin{proof}
We have already seen that metric spaces are hls-normal as hls-spaces, so the result follows from Corollary \ref{TietzeTheorem}.
\end{proof}
The purely large scale version of the above result is just Theorem \ref{DMtheorem}. 

Neighbourhood operators can also be applied to obtain results for small scale/uniform spaces. We will use the definition of uniform space in terms of covers introduced by Tukey \cite{Tukey} (see also \cite{Isbell}), which is equivalent to the original definition in terms of entourages introduced by Weil \cite{Weil} and used in Chapter 2 of Bourbaki's book on general topology \cite{Bourbaki}. A \textbf{uniform space} is a set $X$ equipped with a collection $\mathcal{S}$ of covers of $X$ (which we call ``uniform covers'') satisfying the following axioms:
\begin{itemize}
\item $\{X\}$ is in $\mathcal{S}$,
\item If $\st(\mathcal{U}, \mathcal{U})$ refines $\mathcal{V}$ and $\mathcal{U}$ is in $\mathcal{S}$, then $\mathcal{V}$ is also in $\mathcal{S}$,
\item if $\mathcal{U}$ and $\mathcal{V}$ are elements of $\mathcal{S}$, then there exists an element $\mathcal{W}$ of $\mathcal{S}$ such that $\st(\mathcal{W}, \mathcal{W})$ refines both $\mathcal{U}$ and $\mathcal{V}$.
\end{itemize}

Note the apparent duality with the notion of large scale space. Indeed, uniform spaces are also called \textbf{small scale spaces} in the literature. For a more formal investigation of the connections and duality between large and small scale structures, we refer the reader to \cite{AustinThesis} and \cite{ADH}. A map $f: X \rightarrow Y$ from a uniform space $X$ to a uniform space $Y$ is called \textbf{uniformly continuous} if for every uniform cover $\mathcal{V}$ of $Y$, $f^{-1}(\mathcal{V}) = \{f^{-1}(V) \mid V \in \mathcal{V} \}$ is a uniform cover of $X$. Metric spaces such as $\mathbb{R}$ carry a natural uniform structure consisting of all covers which have positive Lebesgue number. For compact metric spaces, this is just the set of all covers which are refined by an open cover.

\begin{Lemma}
Let $X$ be a uniform space and $f: X \rightarrow [0,1]$ a function. Then $f$ is uniformly continuous if and only if it is neighbourhood continuous with respect to the uniform neighbourhood operator on $X$ and the topological neighbourhood operator on $[0,1]$.
\end{Lemma}
\begin{proof}
$(\Rightarrow)$ is easy to check using Lemma \ref{easyintervals}.

$(\Leftarrow)$ Let $\varepsilon > 0$. Choose a finite number of points $t_1, \ldots, t_n$ in $[0,1]$ such that $0 < t_{k+1} - t_k < \varepsilon/2$ for all $1\leq k \leq n-1$. By neighbourhood continuity, for each of the intervals $[0, t_i]$, we have a uniform cover $\mathcal{U}_n$ of $X$ such that $\st(f^{-1}([0, t_n]), \mathcal{U}_n) \subseteq f^{-1}([0, t_{n+1}))$. Taking a common refinement $\mathcal{V}$ of the $\mathcal{U}_n$, we have that $\mathsf{diam}(V) \leq \varepsilon$ for every $V \in \mathcal{V}$ as required.
\end{proof}

One can check that for a subset $A$ of a uniform space $X$, the uniform neighbourhood operator on $A$ induced by the subspace uniform structure on $A$ coincides with the neighbourhood operator induced by the uniform neighbourhood operator on $X$. Thus we recover the result of Katetov below. 
\begin{Corollary}[Katetov \cite{Katetov}]\label{uniformTietze}
Let $X$ be a uniform space and $A \subseteq X$ a subspace. Then any uniformly continuous function $f: X\rightarrow [0,1]$ extends to a uniformly continuous function $g$ on the whole of $X$.
\end{Corollary}
\begin{proof}
It is enough to show that the uniform neighbourhood operator always satisfies the axiom $\mathsf{(N4)}$. Suppose we have $N, M \subseteq X$ with $N \prec M$ with respect to the uniform neighbourhood operator. Then there is a uniform cover $\mathcal{U}$ in $X$ such that $\st(N, \mathcal{U}) \subseteq M$. By the definition of uniform structure, there is a uniform cover $\mathcal{V}$ such that $\st(\mathcal{V}, \mathcal{V})$ refines $\mathcal{U}$. Then $N \prec \st(N, \mathcal{V}) \prec \st(\st(N, \mathcal{V}), \mathcal{V}) \subseteq M$ as required.
\end{proof}

We should mention some similarities between the work in \cite{Katetov} and the approach to extension theorems via neighbourhood operators in this paper (which was developed independently with large scale spaces in mind). In \cite{Katetov}, Katetov proves a version of his insertion theorem for abstract relations on sets and functions preserving them, with two key examples of such relations being what we call the topological and uniform neighbourhood operators in this paper. From this he is able to obtain the (topological) Katetov-Tong Theorem (Theorem 1 in \cite{Katetov}), as well as the result for uniform spaces given above. In the corrections to \cite{Katetov}, the author notes that some axioms are needed (on the relations) in order to prove the insertion theorem for relations. These axioms (found in Lemma 1 of the corrections) closely resemble axioms $\mathsf{(N3)}$ and $\mathsf{(N4)}$ given in this paper.

\section{The Higson compactification and corona}
The concept of Higson compactification really belongs to hybrid large scale geometry.
For completeness, let's prove the following result:
\begin{Proposition}
 Given a hybrid large scale space $X$ the following conditions are equivalent:
 \begin{itemize}
 \item[(1)] There is a Hausdorff compactification $h(X)$ of $X$ with the property that every continuous slowly oscillating function $f:X\to [0,1]$ extends uniquely over $h(X)$,
\item[(2)]  $X$ is Tychonoff as a topological space.
 \end{itemize}
\end{Proposition}
\begin{proof}
 The implication (1)$\implies$(2) holds for any space $X$ that admits a Hausdorff compactification. 
 To show (2)$\implies$(1) first observe that, given any $x_0\in X$ and given a bounded open neighbourhood $U$ of $x_0$ in $X$, any continuous function $f:X\to [0,1]$ that vanishes outside of $U$ is slowly oscillating. Thus the set of continuous slowly oscillating functions on $X$ separates closed sets from points. It is easy to check that the collection of continuous slowly oscillating functions from $X$ to $[0,1]$ is a subring of the ring of continuous functions from $X$ to $[0,1]$ that is complete with respect to the sup-norm and contains the constant functions. Thus (1) follows from well-known results in compactification theory (see for example Theorem (m) in Section 4.5 of \cite{PorterWoods}).
\end{proof}

In \cite{Roe} (pp.~30--31) the Higson corona of a coarse space $X$ is defined abstractly as a compact space $\nu X$ satisfying
$$C(\nu X)=\frac{B_h(X)}{B_0(X)}.$$ 
Here $B_h(X)$ is the C$^\ast$-algebra of all bounded slowly oscillating (not necessarily continuous) complex-valued functions and $B_0(X)$ is the closed two-sided ideal of functions that ``approach $0$ at infinity'', i.e.~all $f \in B_h(X)$ such that for every $\varepsilon >0$ there is a bounded set $K$ such that $|f(x)| < \varepsilon$ for all $x \notin K$.  It is shown that the geometric realization of the Higson corona, in the case of a (paracompact) proper coarse space, can be obtained as $h(X)\setminus X$, where $h(X)$ is the Higson compactification of $X$, i.e. the compactification corresponding to the algebra of all continuous bounded slowly oscillating functions $X\to [0,1]$.

In case of arbitrary hybrid large scale spaces we can talk about two ways of defining the Higson corona: one as above (using $B_h(X)/B_0(X)$) and the other using continuous slowly oscillating functions, that is via the formula
$$C(\nu X)=\frac{B^c_h(X)}{B^c_0(X)}.$$ 
where $B^c_h$ and $B^c_0$ are the subalgebras of continuous functions in $B_h$ and $B_0$ respectively. One purpose of this section is to show that for normal hls-spaces these definitions are equivalent. There is a natural homomorphism  $\frac{B^c_h(X)}{B^c_0(X)}\to  \frac{B_h(X)}{B_0(X)}$ induced by the inclusion of $B^c_h$ into $B_h$; what we are interested in is when that homomorphism is an isomorphism. 
 
\begin{Theorem}
 If $X$ is hls-normal as a hybrid ls-space, then the natural homomorphism  $\alpha: \frac{B^c_h(X)}{B^c_0(X)}\to  \frac{B_h(X)}{B_0(X)}$ is an isomorphism.
\end{Theorem}
\begin{proof} 
Since $\alpha$ has trivial kernel, it is enough to show that $B_h = B_0 + B^c_h$. Let $f \in B_h$ and let $\mathcal{U}$ be an open scale. Let $A$ be a subset of $X$ which is maximal with the property that no two elements of $A$ are in the same element of $\st(\mathcal{U}, \mathcal{U})$. Then $A$ is a discrete subset of $X$, and no element of $x$ belongs to the closure of more than one element of $A$. Define a map $f'$ from $\mathsf{cl}(A)$ to $[0,1]$ which sends $a' \in A$ to $f(a)$, where $a'$ is in closure of $\{a\}$. Then one checks that $f'$ is slowly oscillating and continuous. By Theorem \ref{TietzeTheorem}, we can extend $f'$ to a continuous slowly oscillating function $g$ on all of $X$. It remains to show that $g-f$ is in $B_0$. Indeed, let $\varepsilon >0$. Then for some bounded set $K$, $\{a,b\} \in U \in \st(\mathcal{U}, \mathcal{U})$ implies that $|g(a) - g(b)| < \varepsilon/2$ and $|f(a) - f(b)| < \varepsilon/2$. Since every element of $X$ is in the same element of $\st(\mathcal{U}, \mathcal{U})$ as some element of $A$, we have that $|f(x) - g(x)| < \varepsilon$ for every $x \notin K$.
\end{proof}

\begin{Proposition}
 Suppose $X$ is a hybrid large scale space whose topology is Tychonoff. Then $X$ is hls-normal if and only if, for each closed subset $Y$ of $X$, its closure $\overline{Y}$ in the Higson compactification $h(X)$ is the Higson compactification of $Y$.
\end{Proposition}
\begin{proof}
The Higson compactification of a closed subset $Y$ of $X$ is completely characterized by the fact that any continuous slowly oscillating complex-valued function on $Y$ extends uniquely to $hY$. If $X$ is hls-normal, then any continuous slowly oscillating function on $Y$ extends to the whole of $X$ by Corollary \ref{TietzeTheorem}, and hence to $hX$ and in particular, to $\overline{Y}$, the closure of $Y$ in $hX$. Uniqueness is easy to check. Conversely, if $\overline{Y}$ is the Higson compactification of $Y$ then any continuous bounded slowly oscillating complex-valued function $f$ on $Y$ extends to a continuous function on $\overline{Y} = hY$. By the classical Tietze Extension Theorem, this extends to a continuous function on $hX$, which when restricted to $X$ is a continuous bounded slowly oscillating function extending $f$.
\end{proof}

\begin{Proposition}
Let $X$ be an hls-space whose topology is Tychonoff. Then the following are equivalent, where for $Y \subseteq X$, $\overline{Y}$ denotes the closure of $Y$ in $hX$:
\begin{itemize}
\item[(1)] $X$ is hls-normal,
\item[(2)] two disjoint closed subsets $A$ and $B$ of $X$ are coarsely separated if and only if $\overline{A} \cap \overline{B} = \varnothing$.
\end{itemize} 
\end{Proposition}
\begin{proof}
(1)$\implies$ (2): Suppose $X$ is hls-normal. If $\overline{A} \cap \overline{B} = \varnothing$ then we can define a continuous map $f$ from $hX$ to $[0,1]$ which sends $\overline{A}$ to $0$ and $\overline{B}$ to $1$. The restriction of $f$ to $X$ is a slowly oscillating function sending $A$ to $0$ and $B$ to $1$. It follows that $A$ and $B$ are coarsely separated. If $A$ and $B$ are coarsely separated, then by Corollary \ref{UrysohnLemmaInHLS}, we can define a slowly oscillating function sending $A$ to $0$ and $B$ to $1$. Extending this to $hX$ we see that we must have $\overline{A} \cap \overline{B} = \varnothing$ as required. 

(2)$\implies$ (1): By Corollary \ref{UrysohnLemmaInHLS} it is enough to produce, for closed subsets $A$ and $B$ of $X$ such that $\overline{A} \cap \overline{B} = \varnothing$, a slowly oscillating continuous map $f$ sending $A$ to $0$ and $B$ to $1$. This can be accomplished by constructing a continuous map sending $\overline{A}$ to $0$ and $\overline{B}$ to $1$ and restricting to $X$. 
\end{proof}

Note that for proper metric spaces, condition (2) in the above proposition follows from Proposition 2.3 in \cite{Dranetal} and plays a crucial role in relating properties of a proper metric space with its Higson corona in various places in the literature (see for example \cite{AustinVirk} or \cite{Dranetal}),

\section{Hybrid structures induced by compactifications}

In this section we discuss hybrid ls-structures related to the work of
 Mine,   Yamashita, and  Yamauchi (see \cite{MY}, \cite{MYY}) who studied properties of the $C_0$-structure
on a locally compact metric space relative to a compact metric compactification. Our next definition generalises that concept.

\begin{Definition}
 Given a closed subset $A$ of a topological space $X$  with empty interior define the large scale structure $LS(X,A)$ on $X\setminus A$
 as follows: a family $\mathcal{U}$ of subsets of $X \setminus A$ is in $LS(X,A)$ if  and only if for each open neighbourhood $U$ of any $a\in A$ in $X$ there is an open neighbourhood
 $V$ of $a$ in $U$ such that $W\in \mathcal{U}$ and $W\cap V\ne \emptyset$ implies $W\subset U$.
\end{Definition}

It is easy to check that this indeed defines an ls-structure. Note that the bounded sets in $X \setminus A$ equipped with the ls-structure $LS(X,A)$ are precisely the subsets of $X \setminus A$ whose closure does not intersect $A$. 

\begin{Proposition}
 Given a closed subset $A$ of a topological space $X$ with empty interior and given a continuous
 function $f:X\setminus A\to Y$ to a complete metric space $Y$, consider the following statements:
 \begin{enumerate}
 \item  $f$ extends continuously over $X$,
\item  $f$ is slowly oscillating with respect to the large scale structure $LS(X,A)$ on $X\setminus A$.
 \end{enumerate}
It is always the case that $(1) \Rightarrow (2)$ and, if each point of $A$ has a countable basis of neighbourhoods and $X$ is Hausdorff, then $(2) \Rightarrow (1)$.
\end{Proposition}
\begin{proof}
$(1) \Rightarrow (2)$: Let $\mathcal{U}$ be an element of $LS(X, A)$ and let $\varepsilon > 0$. For each point $a \in A$, pick a open neighbourhood $V_a$ of $a$ such that $f(V_a)$ has diameter less than $\varepsilon$, and choose an open neighbourhood $W_a$ of $a$ inside $V_a$ such that for all $U\in \mathcal{U}$, $U \cap W_a \neq \varnothing \Rightarrow U \subseteq V_a$. Consider the union $W$ of all the $W_a$. Its complement is a closed subset of $X \setminus A$, hence bounded. Any set $U \in \mathcal{U}$ which intersects $(X \setminus A) \setminus W$ must be contained in an element of $V_a$, so that $f$ is slowly oscillating as required.

$(2) \Rightarrow (1):$ Suppose $f:X\setminus A\to Y$ is continuous and slowly oscillating.
The first issue is to construct an extension $g:X\to Y$ of $f$ and then to show its continuity.
The natural way to define $g(a)$ for $a\in A$ is as the only point belonging to the intersection
of all sets $cl(f(U))$, $U$ being a neighbourhood of $a$ in $X$.
Choose a decreasing sequence $\{U_n\}$ of neighbourhoods of $a$ in $X$.
The intersection
of all sets $cl(f(U_n))$, $n\ge 1$, consists of exactly one point if for each $\epsilon > 0$
there is $M > 0$ such that the diameter of $f(U_n)$ for $n > M$ is smaller than $\epsilon$. Suppose for contradiction that there is a sequence $x_n, y_n\in U_n$ so that $dist(f(x_n),f(y_n)) \ge \epsilon$. Since every bounded set in $X \setminus A$ is contained in some $X \setminus U_n$, the family $\{x_n,y_n\}_{n\ge 1}$ cannot be uniformly bounded because $f$ is slowly oscillating. By definition of the ls-structure $LS(X, A)$, there must exist a $b \in A$ and a neighbourhood $V$ of $b$ such that for every neighbourhood $V' \subseteq V$ of $b$, there is an $n$ for which $\{x_n, y_n\} \cap (X \setminus V)$ and $\{x_n, y_n\} \cap V'$ each have exactly one point. We claim that $a = b$. Indeed, if not, then since $X$ is Hausdorff, we can choose a neighbourhood of $b$ which contains none of the $x_n$, $y_n$, a contradiction. Suppose then that $U_k \subseteq V$. We can choose a neighbourhood $V' \subseteq U_k$ of $a$ such that $V'$ does not contain $x_i$ or $y_i$ for $i \leq k$. It follows that if $\{x_n, y_n\} \cap V' \neq \varnothing$, then $\{x_n, y_n\} \subseteq U_k \subseteq V$. This is a contradiction. Thus $f$ is well-defined, and its continuity is easy to show.
\end{proof}

\begin{Corollary}
 If $X$ is compact Hausdorff, $A$ is a closed subset of $X$ with empty interior whose every point has a countable basis of neighbourhoods
 in $X$, and $LS(X,A)$ is a hybrid large scale space when equipped with the topology induced from $X$,
 then the Higson compactification of $X\setminus A$ equipped with the ls-structure $LS(X,A)$ is exactly $X$.
\end{Corollary}

\begin{Proposition}
 If $X$ is a compact metric space and $A$ is a closed subset of $X$ with empty interior, then $LS(X,A)$ is a hybrid large scale space when equipped with the topology induced from $X$.
\end{Proposition}
\begin{proof}
 Consider the family $\{B(x,d(x))\}_{x\in X\setminus A}$,
 where $d(x)$ is half the distance from $x$ to $A$. It is a scale in $LS(X,A)$.
 \end{proof}

\begin{Proposition}\label{XAcoarsesep}
 Suppose $X$ is a Hausdorff topological space, $A$ is a closed subset of $X$ with empty interior, and each point of $A$ has a countable
 basis of neighbourhoods in $X$. If $LS(X,A)$ is a hybrid large scale space when equipped with the topology induced from $X$,
 then two closed subsets $B$ and $C$ of $X\setminus A$ are coarsely disjoint if and only if their closures
 in $X$ are disjoint.
\end{Proposition}
\begin{proof}
 Suppose $B$ and $C$ are coarsely disjoint but $a\in A$ belongs to $cl(B)\cap cl(C)$.
 Pick sequences $b_n\in B$ and $c_n\in C$, both converging to $a$. We claim that
 $\mathcal{F}:=\{b_n,c_n\}_{n=1}^\infty$ is a uniformly bounded family in $LS(X,A)$. Indeed, let $d \in A$ have open neighbourhood $U$ in $X$. If $d = a$, then we can choose $N > 0$ such that $b_n \in U$ and $a_n \in U$ for all $n > N$. Using the fact that $X$ is Hausdorff, we can choose a smaller neighbourhood $V \subseteq U$ of $d$ which does not contain $a_i$ or $b_i$ for $i \leq N$. Thus if $F \in \mathcal{F}$ intersects $V$, it must be contained in $U$. If $a \neq d$, then we can use the fact that $X$ is Hausdorff to choose an open neighbourhood $V \subseteq U$ of $d$ which contains none of the $a_n$ or $b_n$, so that no element of $\mathcal{F}$ intersects $V$. Thus $\mathcal{F}$ is uniformly bounded. On the other hand, $st(B,\mathcal{F})\cap st(C,\mathcal{F})$ is not bounded because its closure contains $a$. This is a contradiction.
 
 Suppose $B$ and $C$ are closed in $X\setminus A$ and $cl(B)\cap cl(C)=\emptyset$.
 Since any scale of $LS(X,A)$ can be coarsened to an open scale, it suffices to show
 that $st(B,\mathcal{U})\cap st(C,\mathcal{U})$ is bounded for any open scale $\mathcal{U}$ of $LS(X,A)$.
 Suppose, on the contrary, that $a\in A$ belongs to the closure of $st(B,\mathcal{U})\cap st(C,\mathcal{U})$.
 Without loss of generality, we may assume $a\notin cl(B)$.
 Pick a neighbourhood $V$ of $a$ in $X\setminus cl(B)$ such that
 $U\cap V\ne\emptyset$, $U\in \mathcal{U}$, implies $U\subset X\setminus cl(B)$.
 Since $V\cap st(B,\mathcal{U})\ne\emptyset$, there is $U\in \mathcal{U}$
 intersecting both $B$ and $V$. But then $U\subset X\setminus cl(B)$, a contradiction.
\end{proof}

\begin{Corollary}
 Suppose $X$ is a normal topological space, $A$ is a closed subset of $X$ with empty interior, and each point of $A$ has a countable
 basis of neighbourhoods in $X$. If $LS(X,A)$ is a hybrid large scale space when equipped with the topology induced from $X$,
 then it is hls-normal.
 \end{Corollary}
\begin{proof}
Suppose $B$ and $C$ are disjoint, closed, coarsely separated subsets of $X \setminus A$. By Proposition \ref{XAcoarsesep}, the closures of $B$ and $C$ in $X$ are disjoint. Thus the function $f$ from $cl(B) \cup cl(C) \subseteq X$ to $[0,1]$ which sends $cl(B)$ to $0$ and $cl(C)$ to $1$ is well-defined. Since it is continuous, it can be extended to the whole of $X$ by topological normality. Thus the restriction of $f$ to $X \setminus A$ is a slowly oscillating function, and sends $B$ to $0$ and $C$ to $1$.
  \end{proof}
  
 \section{Topological groups as hls-spaces} \label{secgroups}
Let $G$ be a group. Then $G$ admits a natural ls-structure given by all families of subsets $\mathcal{U}$ which refine a family of the form $\{ g \cdot F \mid g \in G\}$ for some finite subset $F$ \cite{DH}. If the group is finitely generated, then this ls-structure coincides with the one given by the word-length metric associated to any finite generating set (see for example \cite{NowakYu} for a definition of this metric). If the group is countable, then this ls-stucture coincides with the unique ls-structure which is induced by a proper left-invariant metric on the group \cite{JSmith}. For a subset $F$ of $G$, we denote the cover $\{g \cdot F \mid g \in G\}$ by $G(F)$. The following lemma gives an explicit formula for starring such covers.

\begin{Lemma}\label{starringingroup}
Let $E$ and $F$ be subsets of $G$. Then we have
\begin{align*}
\st(E, G(F)) & = E\cdot F^{-1} \cdot F \\
\st(G(E), G(F)) & = G(E \cdot F^{-1} \cdot F)
\end{align*}
\end{Lemma}
\begin{proof}
If $x \in \st(E, G(F))$, then there is an $e \in E$ and $g \in G$ such that $e = gf_1$ and $x = gf_2$ for some $f_1,\ f_2 \in F$. Thus $g = ef_1^{-1}$ so $x \in E \cdot F^{-1} \cdot F $ as required. On the other hand, if $x = e f_1^{-1} f_2 \in E \cdot F^{-1} \cdot F $, then $e \in E$ and $\{e = e f_1^{-1} f_1,\ x = e f_1^{-1} f_2 \} \subseteq ef^{-1} \cdot F$, so $x \in \st(E, G(F))$ as required. Since $\st(G(E), G(F))$ is the collection of all sets $\st(g \cdot E, G(F))$, to prove the second equation it suffices to note that $g \cdot (E \cdot F^{-1} \cdot F) =   (g\cdot E) \cdot F \cdot F^{-1}$.
\end{proof}

More generally, if $G$ is a locally compact topological group, then $G$ admits an ls-structure given by all families of subsets which refine $G(K)$ for some compact set $K$. Since the product of two compact subsets in a topological group is again compact, Lemma \ref{starringingroup} shows that this is indeed an ls-structure. Moreover, $G$ together with this structure and the topology given form a hybrid large scale space (the uniformly bounded open cover is just $G(V)$, where $V$ is a precompact neighbourhood of the identity element). We now describe coarse neighbourhoods in the case of a locally compact topological group.

\begin{Lemma}\label{NecessaryAndSufficientForCoarseNbhg}
Suppose $V$ is a precompact, symmetric neighbourhood of the identity element in a locally compact topological group $G$. 
Then the following conditions are equivalent:
\begin{itemize}
\item[(1)] $N$ is a coarse neighbourhood of $U$,
\item[(2)] $U\cdot V\cdot F\cdot V\setminus N$ is precompact for each finite subset $F$ of $G$,
\item[(3)] $U\cdot V\cdot x\cdot V\setminus N$ is precompact for each point $x$ of $G$.
\end{itemize}
\end{Lemma}
\begin{proof}
(1)$\implies$(2). Suppose $N$ is a coarse neighbourhood of $U$ and $F$ is a finite subset of $G$. Enlarge $F$, if necessary, to contain the neutral element $1_G$ of $G$ and be symmetric.
 Consider the uniformly bounded family $\mathcal{U} = G(F \cdot V)$. If $N$ is a coarse neighbourhood of $U$, then there is a precompact set $C$ such that $st(U,\mathcal{U})\subset N\cup C$. By \ref{starringingroup} that implies 
 $U\cdot V\cdot F\cdot F\cdot V\setminus N\subset C$ is precompact.
 In particular, $U\cdot V\cdot F\cdot V\setminus N$ is precompact.
 
 (2)$\iff$(3) is obvious.
  
(2)$\implies$(1) Given a precompact $C\subset G$, find a symmetric finite subset $F$ of $G$ satisfying $C\subset F\cdot V$.
The uniformly bounded family $\mathcal{W}= G(F\cdot V)$ coarsens the cover $\mathcal{C}=G(C)$.
Since (using \ref{starringingroup}) $st(U,\mathcal{W})\setminus N$ is precompact,
so is $st(U,\mathcal{C})\setminus N$.
\end{proof}

\begin{Theorem}\label{groupsnormality}
Let $G$ be a locally compact abelian group. Then the following conditions are equivalent:
\begin{itemize}
\item[(1)] $G$ is hybrid large scale normal as an hls-space,
\item[(2)] $G$ is $\sigma$-compact,
\item[(3)] the ls-structure on $G$ is metrizable, that is, induced by a metric.
\end{itemize}
\end{Theorem}
\begin{proof}
(1) $\Rightarrow$ (2) Suppose $G$ is not $\sigma$-compact. By local compactness, we can pick a countably infinite discrete subset $B$ of $G$. Let $V$ be a precompact symmetric neighbourhood of the identity element. Notice that for any countable set $C$ in $G$, $G$ cannot be generated by $V \cup C$. Thus we can construct an uncountable set $A = \{a_t\}_{t < \omega_1}$ of elements of $G$ indexed by countable ordinals such that for any $t < \omega_1$, $a_t$ is not in the subgroup generated by $B \cup V \cup \{a_r \mid r < t \}$. Note that $B$ is discrete, as is $A$ (since any subset $g \cdot V$ intersects at most one element of $A$), so that every precompact subset of either $A$ or $B$ must be finite. We claim that $A$ has coarse neighbourhood $G \setminus B$. 
First note that $\mathsf{cl}(A) \subseteq \st(A, G(V))$ and $\mathsf{cl}(B) \subseteq \st(B, G(V))$ are disjoint, so that $G \setminus B$ contains an open set which contains the closure of $A$. 
We now show that $G \setminus B$ is a coarse neighbourhood of $A$. Let $x$ be an element of $G$, and consider the set $A\cdot V \cdot x \cdot V \cap B$. If $b_1 = a_1 v_1 x v_1' \in B$ and $b_2 = a_2 v_2 x v_2' \in B$ with $a_1, a_2 \in A$ and $v_1, v_1', v_2, v_2' \in V$, then $a_1$ (resp. $a_2$) is in the subgroup generated by $B \cup V$ and $a_2$ (resp. $a_1$) so we must have $a_1 = a_2$. Thus $A \cdot V \cdot x \cdot V \cap B$ contains only a single point, so by Lemma \ref{NecessaryAndSufficientForCoarseNbhg} we have that $X\setminus B$ is a coarse neighbourhood of $A$.
Suppose for contradiction there is an intermediate coarse neighbourhood $N$ between $A$ and $G \setminus B$. For $b \in B$, let $Z(b) = \{a \in A \mid a \cdot b \notin N\}$. Since $N$ is a coarse neighbourhood of $A$, each $Z(b) \cdot b$, and thus each $Z(b)$, must be precompact, hence finite. Thus the union of all the $Z(b)$ is countable, so there is an $a \in A$ such that $a \cdot b \in N$ for all $b \in B$. But then $a^{-1} \cdot N  \subseteq N \cdot V \cdot a^{-1} \cdot V$ intersects $B$ in an infinite (in particular, not precompact) set, which by Lemma \ref{NecessaryAndSufficientForCoarseNbhg} contradicts the fact that $G\setminus B$ is a coarse neighbourhood of $N$.

(2) $\Rightarrow$ (3) Suppose $G  = \cup_{i=1}^{\infty} K_i$, where the $K_i$ are compact subspaces. If $\mathcal{V}$ is an uniformly bounded open cover, then $G$ is the union of the countable set $\{\st(K_i, \mathcal{V}) \}_{i = 1}^{\infty}$ of precompact open sets. Every compact set is contained in a union of finitely many of the $\st(K_i, \mathcal{V})$. It follows that there is a countable set $C$ of precompact subsets such that each precompact subset is contained in an element of $C$, and that consequently, the ls-structure on $G$ is countably generated, i.e.~metrizable (see Theorem 2.55 of \cite{Roe}).

(3) $\Rightarrow$ (1) Since the ls-structure is metric, $G$ is ls-normal as an ls-space. It is well-known that locally compact Hausdorff groups are topologically normal, so by Theorem \ref{BigNormalityIssue} we have the result.

\end{proof}
 
 \begin{Corollary}
 Let $X$ be the set $\mathbb{R}$ be equipped with the ls-structure coming from the group structure and the discrete topology. Then $X$ is not hls-normal.
 \end{Corollary}
 
 \begin{Question}
 The proof of Proposition \ref{groupsnormality} holds more generally for any group where the ``left-translation structure'' (that is, the ls-structure used here) and the ``right-translation structure'' (that is, the ls-structure generated by families of the form $\{  K \cdot g \mid g \in G\}$ for $K$ compact) coincide: only the last sentence of (1) $\Rightarrow$ (2) needs to be changed. Does Proposition \ref{groupsnormality} hold for all groups?
 \end{Question}

\section{Coarse neighbourhoods and ls-structures}
Since the coarse neighbourhood operator completely determines which maps to $[0,1]$ from an ls-space are slowly oscillating, one might ask to what extend coarse neighbourhoods determine the ls-structure on a set. 

\begin{Definition}
Given any ls-space $X$ with ls-structure $\mathcal{X}$, we define $\mathcal{X}^{cn}$ to be the collection of all families of subsets $\mathcal{U}$ such that for every coarse neighbourhood $N$ of $A$ in $X$, $\st( A, \mathcal{U}) \subseteq N \cup K$ for some bounded set $K$ (that is, bounded in the sense of the original structure $\mathcal{X}$).
\end{Definition}

Clearly we always have $\mathcal{X} \subseteq \mathcal{X}^{cn}$, and if $\mathcal{X}_1 \subseteq \mathcal{X}_2$ are ls-structures on $X$ which induce the same bounded sets, then $\mathcal{X}_1^{cn} \subseteq \mathcal{X}_2^{cn}$. The collection $\mathcal{X}^{cn}$ need not coincide with $\mathcal{X}$ in general, as the following example shows.

\begin{Example}\label{nonex1}
Let $X$ be an infinite set and let $\mathcal{X}$ be the maximal uniformly locally finite ls-structure. Then $N$ is a coarse neighbourhood of $A \subseteq N$ if and only if either $A$ is finite or $N$ is cofinite. Consider a cover of $X$ by infinitely many disjoint finite subsets of unbounded cardinality. One checks that this cover is in $\mathcal{X}^{cn}$, but it is clearly not in $\mathcal{X}$.
\end{Example}

For metric spaces, however, $\mathcal{X}^{cn}$ turns out to be equal to the original (metric) ls-structure.

\begin{Proposition}
Let $X$ be a metric space and $\mathcal{X}$ the associated ls-structure. Then $\mathcal{X} = \mathcal{X}^{cn}$. 
\end{Proposition}
\begin{proof}
We already have $\mathcal{X} \subseteq \mathcal{X}^{cn}$, so suppose for contradiction that there is a family $\mathcal{U}$ in $\mathcal{X}^{cn}$ which is not in $\mathcal{X}$. Since $\mathcal{U}$ is not uniformly bounded, we may choose a sequence $U_n$ of elements of $\mathcal{U}$ and pairs $\{a_n, b_n\} \subseteq U_n$ of points in $X$ such that the $b_n$ are unbounded, and for each $n$, $d(a_n, b_n) > n $. Consider the subset $N = \bigcup_{n=0}^{\infty} B(a_n, n)$. Clearly it is a coarse neighbourhood of $A = \{a_n \mid n \in \mathbb{N} \}$, but $\st(A, \mathcal{U}) \cap X \setminus N$ contains the $b_n$s, so it is unbounded. 
\end{proof}

More generally, one may ask when $\mathcal{X}^{cn}$ is an ls-structure. As shown below, for ls-normal ls-spaces $\mathcal{X}^{cn}$ turns out to be an ls-structure. Note that the maximal uniformly locally finite structure is ls-normal, so even if $\mathcal{X}^{cn}$ is an ls-structure, it need not coincide with the original structure $\mathcal{X}$ as Example \ref{nonex1} above shows.

\begin{Proposition}
Let $X$ be an ls-space which is ls-normal. Then $\mathcal{X}^{cn}$ is an ls-structure.
\end{Proposition}
\begin{proof}
Let $\mathcal{U}$ and $\mathcal{V}$ be elements of $\mathcal{X}^{cn}$ and let $A$ have coarse neighbourhood $N$. By normality, there are intermediate coarse neighbourhoods $A \prec L \prec M \prec N$. Then $A_1 = \st(A, \mathcal{V})$ is contained in $L \cup K$ for some bounded subset $K$. In particular, $M$ is a coarse neighbourhood of $\st(A, \mathcal{V})$. Similarly, $A_2 = \st(A_1, \mathcal{U})$ has coarse neighbourhood $N$. Finally, $\st(A_2, \mathcal{V}) = \st(A, \st(\mathcal{U}, \mathcal{V}))$ is contained in $N \cup K'$ for some bounded set $K'$, which gives the result.
\end{proof}

\begin{Proposition}
Let $X$ and $Y$ be metric spaces and let $f: X \rightarrow Y$ be a map which sends bounded sets in $X$ to bounded sets in $Y$ and which is proper (that is, the inverse image of a bounded set is bounded). Then $f$ is ls-continuous if and only if for every subset $A$ of $Y$ and every coarse neighbourhood $N$ of $A$, $f^{-1}(N)$ is a coarse neighbourhood of $f^{-1}(A)$.
\end{Proposition}
\begin{proof}
$(\Rightarrow)$ Suppose $N$ is a coarse neighbourhood of $A \subseteq Y$. Then for every uniformly bounded cover $\mathcal{U}$ of $X$, the image of $\st(f^{-1}(A), \mathcal{U})$ is contained in $\st(A, f(\mathcal{U}))$, which in turn is contained in $N \cup K$ for some bounded set $K$. It follows, since $f$ is coarse, that $f^{-1}(N)$ is a coarse neighbourhood of $f^{-1}(A)$.

$(\Leftarrow)$ Suppose that $f$ is not ls-continuous. Then there is some $R > 0$ such that the set
$$
\{d(f(x),f(x')) \mid (x, x') \in X\times X,\ d(x, x') < R \}
$$
is unbounded. In particular, we may choose a sequence of pairs of points $(a_n, b_n)_{n \in \mathbb{N}}$ in $X$ such that $d(a_n, b_n) < R$ and $d(f(a_n), f(b_n)) > n$ for every $n$. Because $f$ sends bounded sets to bounded sets, the $a_n$ and $b_n$ cannot all be contained in a single bounded set $K$, since otherwise all the $f(a_n)$ and $f(b_n)$ would be contained in the bounded set $f(K)$. Since each $a_n$ is distance at most $R$ from $b_n$, we can moreover say that neither of the sequences $(a_n)_{n \in \mathbb{N}}$ and $(b_n)_{n \in \mathbb{N}}$ are bounded. Thus neither of the sequences $(f(a_n))_{n \in \mathbb{N}}$ and $(f(b_n))_{n \in \mathbb{N}}$ are bounded, because $f$ is a proper map. 

We now choose a subsequence of $(a_n, b_n)$ for which the images of the points are sufficiently spread out. To do so, we pick a base point $y_0 \in Y$ and set $\phi(0) = 0$. The induction proceeds as follows: for $i \in \mathbb{N}$, let 
$$p(i) = \max \left\lbrace \{ d(y_0, f(a_{\phi(k)})) \mid k \leq i \} \cup \{d(y_0, f(b_{\phi(k)})) \mid k \leq i \}  \right\rbrace+ i + 1$$
and choose $\phi(i+1) \geq i+1$ so that $f(a_{\phi(i+1)})$ and $f(b_{\phi(i+1)})$ are both outside the bounded set $B(y_0, p(i))$. Thus we obtain a subsequence $(a_{\phi(n)}, b_{\phi(n)})_{n \in \mathbb{N}}$ with the property that for every $k, i \in \mathbb{N}$, $f(b_{\phi(k)})$ is distance at least $i-1$ from $f(a_{\phi(i)})$.

Clearly $N = \bigcup_{i=0}^\infty B(f(a_{\phi(n)} ), n-1)$ is a coarse neighbourhood of $A = \{f(a_\phi(n)) \mid n \in \mathbb{N} \}$. We claim, however, that $f^{-1}(N)$ is not a coarse neighbourhood of $f^{-1}(A)$. Indeed, if $\mathcal{U}$ is the uniformly bounded cover of $X$ by balls of radius $R$, then $\st(f^{-1}(A), \mathcal{U})$ contains all the $b_{\phi(n)}$. On the other hand, $X \setminus f^{-1}(N)$ also contains each $b_{\phi(n)}$ by construction. Since the set $\{b_{\phi(n)} \mid n \in \mathbb{N}\}$ is unbounded, we have that $\st(f^{-1}(A), \mathcal{U}) \cap (X \setminus f^{-1}(N)))$ is unbounded, so $f^{-1}(N)$ is not a coarse neighbourhood of $f^{-1}(A)$.
\end{proof}


\begin{thebibliography}{99}

\bibitem{AustinThesis} K. Austin, \emph{Geometry of scales}, PhD Thesis, University of Tennessee, 2015.

\bibitem{ADH} K. Austin, J. Dydak and M. Holloway, \emph{Scale Structures and C*-algebras}, (preprint) arXiv:1602.07301.

\bibitem{AustinVirk} K. Austin and \u{Z}. Virk, \emph{Higson compactification and dimension raising}, Topology and its Applications 215, 2017, 45--57.

\bibitem{BellDranAsym} G. Bell and A. Dranishnikov, \emph{Asymptotic dimension}, Topology and its Applications 155.12, 2008, 1265--1296.

%\bibitem{BellDran} G. Bell and A. Dranishnikov, {\em On asymptotic dimension of groups}, Algebraic and Geometric Topology 1, 2001, 57--71. 

\bibitem{Bourbaki} N. Bourbaki, \emph{General Topology}, Chapters 1--4, Springer-Verlag, N. Y., 1980.

%\bibitem{Brezis} H. Brezis, {\em Functional analysis, Sobolev spaces and partial differential equations}, Springer Science \& Business Media, 2010.

%\bibitem{BrodDyd} N. Brodskiy, J. Dydak, M. Levin and A. Mitra, \emph{A Hurewicz theorem for the Assouad-Nagata dimension}, Journal of the London Mathematical Society 77.3, 2008, 741--756.

%\bibitem{CDV2} MM. Cencelj, J. Dydak, A. Vavpeti\v c, \emph{ Large scale versus small scale}, in Recent Progress in General Topology III, Hart, K.P.; van Mill, Jan; Simon, P (Eds.) (Atlantis Press, 2014), pp.165--204

\bibitem{Czaszar}  \'{A}. Cz\'{a}sz\'{a}r, \emph{General Topology}, Akad\'{e}miai Kiad\'{o}, Budapest, 1978.

\bibitem{Dranetal} A. Dranishnikov, J. Keesling and V.~V.~Upenskij, {\em On the Higson corona of uniformly contractible spaces}, Topology 37 No. 4, 1998, 791--803.

\bibitem{DyPU} J. Dydak, {\em    Partitions of unity},    Topology Proceedings
  27, 2003,  125--171.

\bibitem{DH} J. Dydak and C.S. Hoffland, \emph{An Alternative Definition of Coarse Structures}, Topology and its Applications 155, 2008, 1013--1021.

\bibitem{DM} J. Dydak and A. Mitra, \emph{Large scale absolute extensors}, Topology and its Applications 214, 2016, 51--65.

%\bibitem{Eilen} S. Eilenberg, {\em Sur les transformations continues d'espaces m\^{e}triques compacts}, Fund. Math. 22, 1954, 292--296.

%\bibitem{En} R.Engelking, {\em General Topology},Heldermann Verlag Berlin (1989).

\bibitem{Gromov93} M. Gromov, \emph{Asymptotic invariants of infinite groups}, Geometric group theory vol. 2 (Sussex, 1991), London Math. Soc. Lecture Note Ser. vol. 182, Cambridge University Press, Cambridge, 1993, 1--295.


%\bibitem{HigsonRoeYu} N.~Higson, J.~Roe and G.~Yu, A coarse Mayer-Vietoris principle, Mathematical Proceedings of the Cambridge Philosophical Society Volume 114 no. 01, Cambridge University Press, 1993.

\bibitem{HolgateSlapal} D. Holgate and J. \u{S}lapal, \emph{Categorical neighborhood operators}, Topology and its Applications 158, 2011, 2356--2365.

\bibitem{Isbell} J. R. Isbell, \emph{Uniform Spaces}, Mathematical Surveys, vol. 12, American Mathematical Society, Providence, RI, 1964.

\bibitem{Katetov} M. Katetov, \emph{On real-valued functions in topological spaces}, Fund. Math. 38, 1951, 85--91; Correction to ``On real-valued functions in topological spaces'', Fund. Math. 40, 1953, 203--205.


\bibitem{MY} K. Mine and A. Yamashita, {\em Metric compactifications and coarse structures}, (preprint) arXiv:1106.1672.

\bibitem{MYY} K. Mine, A. Yamashita, and T. Yamauchi, {\em $C_0$ Coarse Structures on Uniform Spaces}, Houston Journal of Mathematics 41, 2015, 1351--1358.

%\bibitem{MiyVirk} T. Miyata, Z. Virk, \emph{Dimension raising maps in a large scale}, Fundamenta Mathematicae 223, 2013, 83-97.

\bibitem{Mu} J. Munkres, {\em Topology}, Prentice Hall, 2000, 2nd edition.

\bibitem{NowakYu} P.W. Nowak and G. Yu, {\em Large scale geometry}, EMS Textbooks in Mathematics, European Mathematical Society (EMS), Zurich, 2012. 

\bibitem{PorterWoods} J. R. Porter and R. G. Woods, \emph{ Extensions and absolutes of Hausdorff spaces}, Springer Science \& Business Media, 2012.

\bibitem{Roe} J. Roe, {\em Lectures in Coarse Geometry}, University Lecture Series 31, American Mathematical Society, Providence, RI, 2003. 

\bibitem{RoeIndex} J. Roe, \emph{Index theory, coarse geometry, and topology of manifolds}, CBMS Regional Conference Series in Mathematics Volume 90, Published for the Conference Board of the Mathematical Sciences, Washington, DC, 1996.

\bibitem{Sako} H. Sako, \emph{Property A for coarse spaces}, (preprint) arXiv:1303.7027.

\bibitem{JSmith} J. Smith, {\em On asymptotic dimension of countable abelian groups}, (preprint) arXiv:math/0504447.

%\bibitem{WeiConn} T. Weighill, On spaces with connected Higson coronas, arXiv preprint arXiv:1602.06991, 2016.

\bibitem{Tukey} J. W. Tukey, \emph{Convergence and uniformity in topology},  Ann. Math. Studies no. 1, Princeton University Press, 1940. 

\bibitem{Weil} A. Weil, {\em Sur les espaces \`{a} structure uniforme et sur la topologie g\'{e}n\'{e}rale}, Actualit\'{e}s Sci. Ind. 551, Paris, 1937.
 
%\bibitem{Whyburn} G. T. Whyburn, {\em Open mappings on locally compact spaces}, Memoirs Amer. Math. Soc. 1, 1950.

\bibitem{Wright} N. Wright, \emph{$C_0$ coarse geometry}, PhD thesis, Penn State, 2002.

\end{thebibliography}
\end{document}